\documentclass[11pt]{amsart}
\usepackage{amsmath,amssymb,latexsym,soul,cite,mathrsfs,amsfonts}
\usepackage{float}
\usepackage{xcolor}
\usepackage{color,enumitem,graphicx}
\usepackage[colorlinks=true,urlcolor=blue,
citecolor=blue,linkcolor=blue,linktocpage,pdfpagelabels,
bookmarksnumbered,bookmarksopen]{hyperref}
\usepackage[english]{babel}
\usepackage{outline}
\usepackage{comment}
\usepackage[left=2.9cm,right=2.9cm,top=2.8cm,bottom=2.8cm]{geometry}
\usepackage[hyperpageref]{backref}
\usepackage[colorinlistoftodos]{todonotes}
\makeatletter
\providecommand\@dotsep{5}
\def\listtodoname{List of Todos}
\def\listoftodos{\@starttoc{tdo}\listtodoname}
\makeatother
\usepackage{verbatim}
\numberwithin{equation}{section}

\newtheorem{thm}{Theorem}[section]
\newtheorem{prop}[thm]{Proposition}
\newtheorem{lem}[thm]{Lemma}

\newtheorem{rem}{Remark}
\newtheorem{definition}{Definition}
\newtheorem{ex}{Example}
\newenvironment{step}[1]{\par\noindent\textbf{Step 1.}\space#1}{}
\newenvironment{step2}[1]{\par\noindent\textbf{Step 2.}\space#1}{}
\newenvironment{case}[1]{\par\noindent\textbf{Case 1.}\space#1}{}
\newenvironment{case2}[1]{\par\noindent\textbf{Case 2.}\space#1}{}

\newcommand{\R}{\mathbb{R}}
\newcommand{\h}{H^1_0(\Omega)}
\newcommand{\hr}{H^1_{0, rad}(\Omega)}
\newcommand{\1}{\frac{1}{2}}

\title [ Non-resonant nonlocal problem with critical exponential nonlinearity]{Nonlocal problem with critical exponential nonlinearity of convolution type: A non-resonant case}

\author[Suman Kanungo]{Suman Kanungo}

\address[]{\newline\indent
	Department of Mathematics
	\newline\indent 
	Indian Institute of Technology Bhilai
	\newline\indent
	491002, Durg, Chhattisgarh, India}
\email{\href{sumankau@iitbhilai.ac.in}{sumankau@iitbhilai.ac.in}}

\author[Pawan Kumar Mishra]{Pawan Kumar Mishra}
\address[]{\newline\indent
	Department of Mathematics
	\newline\indent 
	Indian Institute of Technology Bhilai
	\newline\indent
	491002, Durg, Chhattisgarh, India}
\email{\href{pawan@iitbhilai.ac.in}{pawan@iitbhilai.ac.in}}
\subjclass[2010]{Primary 35J15, 35J20, 35J60}
\keywords{Choquard equations, Critical exponential growth, Weighted Trudinger-Moser inequality, Non-resonant}

\begin{document}
	\begin{abstract}
		In this paper, we study the following class of weighted Choquard equations
		\begin{align*}
			-\Delta u =\lambda u + \Bigg(\displaystyle\int\limits_\Omega \frac{Q(|y|)F(u(y))}{|x-y|^\mu}dy\Bigg) Q(|x|)f(u) ~~\textrm{in}~~ \Omega~~ \text{and}~~
			u=0~~  \textrm{on}~~ \partial \Omega,
		\end{align*}  
		where $\Omega \subset \R^2$ is a bounded domain with smooth boundary, $\mu \in (0,2)$ and $\lambda >0$ is a parameter. We assume that $f$ is a real-valued continuous function satisfying critical exponential growth in the Trudinger-Moser sense, and $F$ is the primitive of $f$. Let $Q$ be a positive real-valued continuous weight, which can be singular at zero. Our main goal is to prove the existence of a nontrivial solution for all parameter values except when $\lambda$ coincides with any of the eigenvalues of the operator $(-\Delta, \h)$.
		
	\end{abstract}
	\maketitle
	\section{Introduction}
	\label{sec1}
	This paper focuses on the existence of nontrivial solutions for the following class of weighted Choquard-type equations
	\begin{equation}
		\label{1.1}
		\left\{
		\begin{aligned}
			-\Delta u &=\lambda u + \Bigg(\displaystyle\int\limits_\Omega \frac{Q(|y|)F(u(y))}{|x-y|^\mu}dy\Bigg) Q(|x|)f(u) &&\mbox{in} \ \Omega\\
			u&=0  &&\mbox{on} \ \partial \Omega
		\end{aligned}
		\right.
	\end{equation}  
	where $\Omega \subset \R^2$ is a bounded domain with smooth boundary, $0<\mu<2$, $\lambda$ is a positive parameter but not equal to any of the eigenvalues of the operator $(-\Delta, \h)$ and $Q$ is a continuous weight but can have a singularity at the origin. The nonlinearity $f$ is a continuous function satisfying critical exponential growth along with some suitable conditions specified later, and $F(s) = \int\limits_0^s f(t)dt$. 
	
	The investigation of equation \eqref{1.1} is inspired by the following equation:
	\begin{equation}
		\label{C}
		-\Delta u + V(x)u = \Bigg(\int\limits_{\R^N} \frac{|u(y)|^p}{|x-y|^\mu}dy\Bigg) |u|^{p-2}u  \ \ \text{ in } \R^N,
	\end{equation}	
	where $N \geq 3$, $0< \mu< N$, $p\geq 2$ and $V$ is a real valued continuous potential. The equation \eqref{C}, commonly referred to as the Choquard or Hartree-type equation, arises in various physical scenarios. In 1954, Pekar \cite{Pekar} was the first to introduce the equation \eqref{C} with $N=3, \mu=1, p=2$. He used this equation to study the polaron at rest in the quantum field theory. As discussed in \cite{Lieb} by Lieb, P. Choquard also studied \eqref{C} in an approximation to Hartree-Fock's theory of one component plasma to describe an electron trapped in its hole (see \cite{Lieb2} for physical context).
	
	For the last few decades, many authors have studied the existence and qualitative behavior of solutions for the class of problem \eqref{C}.  In 1977, Lieb \cite{Lieb} established the existence and uniqueness of the solution  by using symmetric decreasing rearrangement inequalities with $N=3, \mu=1, p=2$, and $V$ as a positive constant. In 1980, Lions \cite{Lions} showed the existence of infinitely many radially symmetric solutions for the equation \eqref{C} with $N=3, \mu=1, p=2$ and $V(x)\equiv\lambda>0$. In \cite{Moroz}, authors studied the regularity and positivity of the ground state solutions for the case $\frac{N-2}{2N-\mu}< \frac{1}{p}<\frac{N}{2N-\mu}$ with $V(x)\equiv 1$ in \eqref{C}. They established that all positive ground states are radially symmetric and have decaying properties. For a thorough and insightful review of the Choquard equations of type \eqref{C}, we refer \cite{Ackermann, Alves, Alves3, Biswas1, S.Li, Moroz1, Moroz2, Moroz3, Radulescu} along with the references therein. For other related nonlocal problems dealing  variational methods, we refer to \cite{Ding, Dou, He, Liu, ZShen} and their references.
	
	To deal with the problem variationally, the nonlocal term present in the equation \eqref{1.1} is treated with the help of the following result introduced in \cite{Stein}.
	\begin{prop}[Hardy-Littlewood-Sobolev inequality]
		\label{Hardy}
		Let $s, r > 1$ and $0 < \mu < N$ with $ \frac{1}{s} + \frac{\mu}{N} + \frac{1}{r} = 2 $. Let $f \in L^{s}(\R^N)$ and $h \in L^{r}(\R^N)$. There exists a sharp constant $C(s,N, \mu, r)$, independent of $f, h$, such that
		\begin{equation*} 
			\int\limits_{\R^N} \int\limits_{\R^N} \frac{f(x)h(y)}{|x-y|^{\mu}}dx dy\leq C(s,N, \mu, r)\|f\|_{L^s(\R^N)}\|h\|_{L^r(\R^N)}. 
		\end{equation*} 
	\end{prop}
	We give the following definition of the solution of the problem in \eqref{1.1}.
	\begin{definition}
		We say that a function $u \in H^1_0(\Omega)$ to be a weak solution of \eqref{1.1} if
		\[ \int\limits_\Omega \nabla u \nabla \phi\ dx= \lambda \int\limits_\Omega u \phi\ dx +\int\limits_\Omega \Bigg( \int\limits_\Omega \frac{Q(|y|)|F(u(y))}{|x-y|^{\mu}}dy \Bigg) Q(|x|)f(u(x))\phi(x)\ dx,\]
	\end{definition}
	for all $\phi \in H^1_0(\Omega)$.  
	
	The problem has a variational structure in the sense that the weak solutions of \eqref{1.1} are the critical points of the energy functional $J: \h \rightarrow \R$ associated to \eqref{1.1} and defined as
	\begin{equation}
		\label{J}
		J(u) = \frac{1}{2} \int\limits_\Omega |\nabla u|^2 \ dx - \frac{1}{2}\int\limits_\Omega \lambda u^2 \ dx -\frac{1}{2} \int\limits_\Omega \Bigg( \int\limits_\Omega \frac{Q(|y|)|F(u(y))}{|x-y|^{\mu}}dy \Bigg) Q(|x|)F(u(x)) \ dx.
	\end{equation}

	The case when $N=2$ is special as the critical Sobolev exponent $2^*$ becomes $\infty$. It is well known that for bounded domains $\Omega \subset \mathbb{R}^2$, the corresponding Sobolev embeddings $H^{1}_0(\Omega) \hookrightarrow L^p(\Omega)$ with $1 \leq p < +\infty$ holds but $H^{1}_0(\Omega) \not\hookrightarrow L^\infty(\Omega)$. To overcome the problem of finding an optimal space in Sobolev embedding, the Trudinger-Moser inequality \cite{{Moser},{Trudinger}} can be seen as a suitable alternative. It provides an embedding of $\h$ space into orlicz spaces and is stated as below:
	
	If $\alpha>0$ and $u \in H^1_0 (\Omega)$ then $\int\limits_\Omega e^{\alpha u^2} dx < +\infty.$
	Moreover, 
	\[ \sup_{u\in H^1_0(\Omega), \|u\| \leq 1}\int\limits_\Omega e^{\alpha u^2}dx \leq C |\Omega| \ \ \text{if} \ \ \alpha \leq 4\pi,\] 
	where $C= C(\alpha)>0$ and $|\Omega|$ denotes Lebesgue measure of $\Omega$.
	
	In \cite{Lion}, Lions established a generalized version of the above Trudinger-Moser inequality. 
	Let $\{u_n\} \subset H^1_0(\Omega)$ be a sequence satisfying $\|u_n\|= 1$ for all $n\in \mathbb{N}$ and $u_n \rightharpoonup u_0 $ in $H^1_0(\Omega)$, $0< \|u_0\| <1$, then for all $ 0< p<\frac{4\pi}{1-\|u_0\|^2}$, we have
	\[\sup_{n \in \mathbb{N}} \int\limits_{\Omega} e^{ p |u_n|^2}  \,dx < +\infty. \quad\]
	
	Inspired by the Trudinger-Moser inequality, we can define a notion of criticality, which was introduced by Adimurthi and Yadava \cite{Adimurthi}. It is also discussed in de Figueiredo, Miyagaki, and Ruf \cite{Figueiredo}. We say that a function $f$ has critical exponential growth if there exists a constant $\alpha_0 >0$ such that
	\begin{equation} 
		\label{ceg}
		\lim_{|s| \rightarrow +\infty} \frac{|f(s)|}{e^{\alpha s^2}}=\begin{cases}
			0,& \forall \alpha > \alpha_0,\\
			+\infty, & \forall \alpha < \alpha_0.
		\end{cases}
	\end{equation}
	
	Recently, the nonlocal Choquard-type equation with exponential critical growth in $\R^2$ was explored in \cite{Cassani, Alves2}, where the focus was on establishing the existence of a ground state solution for the nonlocal equation:
	\begin{equation} 
		\label{1.1.2} 
		-\Delta u + W(x)u = \Bigg( \int\limits_{\R^2} \frac{F(u(y))}{|x-y|^\mu}dy \Bigg) f(u), \ \ x \in \R^2,
	\end{equation}
	where the potential $W(x)$ is periodic and bounded from below. Later, the authors in \cite{Shen} studied the problem \eqref{1.1.2} involving a weight function $Q$ and a potential $V$ decay to zero at infinity. Regarding other findings related to the Choquard equation with exponential critical growth, we refer to \cite{Alves1, Arora, Biswas} and references therein. On the other hand, Ribeiro \cite{Bruno} discussed the local version of problem \eqref{1.1} with $Q(|x|)\equiv1$ and $\lambda=\lambda_k $, where $\lambda_k$ represents the $k^{th}$ eigenvalue of $(-\Delta, \h)$ for $k\geq 2$, specifically addressing the resonant case (see \cite{Li}, for the nonlocal case). For the nonlocal operator with mixed boundary conditions, interested readers can refer to \cite{Bisci}. This paper aims to study the nonlocal equation \eqref{1.1} with $\lambda$ that does not coincide with any eigenvalue of $(-\Delta, \h)$ (i.e., the non-resonant case). Hence, our result completes the picture left open in \cite{Bruno} for nonlocal cases. Furthermore, we address nonlinearities involving a weight function $Q(|x|)$, which can be singular at zero.
	
	Precisely, the following assumption is satisfied by the weight function $Q$:
	\begin{align}
		\tag{Q} \label{Q}
		&Q\in C(0, \infty), Q(r) >0 \text{ for } r>0
		\text{ and there exists } b_0 > -\frac{4-\mu}{2}, b\in \R, C_0>0 \text{ and } R>0\\
		&\text{ such that } 0< \liminf_{r\to 0^+} \frac{Q(r)}{r^{b_0}}\leq \limsup_{r \rightarrow 0^+} \frac{Q(r)}{r^{b_0}} < \infty, \text{ and } Q(r) \leq C_0 r^b \ \forall \ r \geq R.\nonumber
	\end{align}
	
	\begin{ex}
		The standard example of a weight function is given by $Q(|x|) = |x|^{b_1}$ satisfying \eqref{Q} with $b_0= b= b_1 > -\frac{4-\mu}{2}$.
	\end{ex}
	\begin{ex}
		In \cite{Felli}, Ambrosetti, Felli, and Malchiodi used the following weight to study nonlinear Shr\"{o}dinger equations,
		\[0< Q(|x|) \leq \frac{A}{1+|x|^{b_2}},\] 
		this weight verifies the condition \eqref{Q} with constants $A>0$ and $b_2 \geq 0$.
	\end{ex}
	
	Note that, compared to the equation \eqref{1.1.2}, the major challenge is due to the presence of the weight function $Q$ in the nonlocal term in equation \eqref{1.1}, which can possess singularity at zero. This term adds more complex computations, introduces mathematical difficulties, and makes the problem both challenging and exciting. Consequently, we need to prove a variant of the Sobolev embedding theorem and the Trudinger-Moser inequality for our scenario (see Lemma \ref{lem3} and Lemma \ref{lem2.2}).
	
	Inspired by the works \cite{Bonheure,Calanchi,Figueiredo1,Figueiredo2, Sandeep} and to investigate equation \eqref{1.1}, we develop a version of the Sobolev embedding and the Trudinger-Moser inequality, which will play a crucial role in our results. In \cite{Sandeep}, Adimurthi and Sandeep derived the following singular version of the Trudinger-Moser inequality for a bounded domain $\Omega$:
	\begin{equation}
		\label{2.5}
		\sup_{ u \in \h, \|u\|\leq 1} \int\limits_\Omega |x|^\beta e^{\alpha |u|^2}dx < \infty \ \text{ if and only if } \alpha \leq 4\pi\Big(1+\frac{\beta}{2}\Big), -2< \beta \leq 0.
	\end{equation}
	On the other hand, the $\beta >0$ case is studied in \cite{Bonheure,Calanchi,Figueiredo1,Figueiredo2}. In which, authors have improved the range of $\alpha$, when $u$ belongs to the radial space, i.e. 
	\begin{equation} 
		\label{2.6}
		\sup_{u\in H^1_{0,rad}(B_1), \|u\|\leq 1} \int\limits_{B_1} |x|^{\beta}  e^{\alpha u^2} dx < \infty, \ \forall \ \alpha \leq 4\pi\Big(1+\frac{\beta}{2}\Big),\ \beta> 0
	\end{equation}
	where $B_1$ is a unit ball, $\hr$ is the closure of $C_{0, rad}(\Omega)$ and $C_{0,rad}(\Omega)$ is the set of all elements of $C_{0}^\infty(\Omega)$ that are radial.
	
	The above weighted Trudinger-Moser inequality shows that to improve the range of $\alpha$, we need to work in the radial space $\hr$ rather than the Sobolev space $\h$. Then, one might inquire whether a minimizer of associated energy in $\hr$ solves equation \eqref{1.1}, i.e., whether a critical point of the energy functional $J$ associated to \eqref{1.1} over $\hr$ can be a critical point of that energy functional $J$ over $\hr$? The principle of symmetric criticality theory provides a positive answer to this question under some additional requirements on energy functional. 
	\subsection{Principle of symmetric criticality} 
	Let $ \mathcal{O}(N) $ denote the group of all orthogonal matrices on $ \mathbb{R}^N $ and $H$ be any closed subgroup of $ \mathcal{O}(N) $. Let $(X, \|\cdot\|_X)$ be a Banach space, an action of a group $ H $ on $ X $ is defined by a smooth continuous map
	\[
	*: H \times X \to X \text{ which maps } (h,u) \mapsto h*u,
	\]
	such that the following conditions will hold
	\[ 1*u = u, \ \ (h*g)*u = h* (g*u),\ \ u \mapsto h*u \text{ is linear. }\]
	If $\|h*u\|_X = \|u\|_X$, for all $h\in H$ and $u\in X$, then we say that the action $*$ is isometric. A function $u\in X$ is called $H$-invariant if, and only if, $h*u = u$. Moreover, the set of all $H$-invariant functions in $X$ is denoted by $Fix_H(X)$, i.e.
	\[ Fix_H(X)= \{ u \in X : h*u=u, \ \forall \ h \in H\}.\]
	We use the following version of Palais principle of symmetric criticality \cite{Palais}  due to Kobayashi and \^Otani \cite{Kabayashi}.
	\begin{thm}[Principle of symmetric criticality]
		\label{thm1.3}
		Let $X$ be a reflexive and strictly convex Banach space and $H \subset \mathcal{O}(N)$ be a group that acts on $X$ linearly and isometrically. If $J : X \to \R$ is an $H$-invariant, $C^1$ functional, then 
		\[ (J|_{Fix_H(X)})'(u)=0 \text{ implies that } J'(u)=0 \text{ and } u \in \h.\]
	\end{thm}
	Note that, if we consider $X= \h$, $ H = \mathcal{O}(N) $, where $\Omega \subset \R^N$, and the action is defined as the standard linear isometric map given by 
	\[ h u (x) = u (h^{-1}x), \ \forall \  x \in \Omega, \ \forall \ h \in \mathcal{O}(N),\]
	then the set of all invariant functions of $\h$, with respect to $\mathcal{O}(N)$ correspond to the space of all radial functions in $\h $, i.e. $ Fix_{\mathcal{O}(N)}(\h)= \hr$.
	
	\begin{definition}
		Let $ \Omega $ be a domain in $ \mathbb{R}^N $. If $ H(\Omega) = \Omega $, then we say $\Omega $ is $ H $-invariant. If $\Omega$ is $H$-invariant and a function $f: \Omega \to \mathbb{R}$ is defined by $f(h(x))=f(x), \ \forall \ h \in H$, then $f$ is called $H$-invariant.
	\end{definition}
	
	Since the space $\h$ is reflexive and strictly convex Banach space, to apply the principle of symmetric criticality, it is sufficient to show the functional $J$ is $\mathcal{O}(2)$-invariant, where $\mathcal{O}(2)$ denotes the group of all orthogonal matrices on $\R^2$. For this, we need to assume that $\Omega$ is invariant with respect to $\mathcal{O}(2)$.
	
	\subsection{Spectral properties of the Laplacian operator}
	Consider the following eigenvalue problem
	\begin{align}  \label{1.5}
		-\Delta u = \lambda u \ \text{ in } \ \Omega, \ u=0 \ \text{ on } \ \partial \Omega. 
	\end{align}
	It is well known that we get a sequence of eigenvalues of the problem \eqref{1.5} and we denote it as $0<\lambda_1<\lambda_2 \leq \lambda_3 \leq \cdots \leq \lambda_k \leq \cdots $ with $\lambda_k \rightarrow \infty$ as $k \rightarrow \infty$.
	The eigenfunctions $\{ \phi_k\}_{k \geq 1}$ corresponding to each $\lambda_k$ forms an orthonormal basis for $L^2(\Omega)$ and an orthogonal basis for $H^1_0(\Omega)$. To prove the existence results, we need to decompose the space $H^1_0(\Omega)$ as:
	\[ H^1_0(\Omega) = H_k(\Omega) \oplus H_k^{\perp} (\Omega), \text{where } H_k(\Omega) = span\{ \phi_1, \phi_2, \cdots, \phi_k\}, \]
	and the orthogonal complement has to be taken with respect to the scalar product $\langle \cdot,\cdot \rangle$ defined on $\h$. The following characterization of eigenvalues is shown in \cite{Mcowen}
	\begin{align} 
		\lambda_1 &= \min_{u \in H^1_0(\Omega) \setminus \{0\}} \frac{\|u\|^2}{\|u\|_2^2} \label{1.6}\\
		\lambda_{k+1} &=  \min_{u \in H_k^\perp (\Omega) \setminus \{0 \}}
		\frac{\|u\|^2}{\|u\|_2^2}   \text{ ~~~  for } k \geq 1,  \label{1.7}
	\end{align}
	where $\|u\|_2 = \Bigg( \int\limits_\Omega |u|^2dx\Bigg)^{\frac{1}{2}}$ is the norm in $L^2(\Omega)$. Similar to the characterization of $\lambda_{k+1}$ in \eqref{1.7}, we can show that $\lambda_k$ can also be characterized as
	\begin{equation}
		\lambda_k = \max_{u \in H_k(\Omega) \setminus \{0\} } \frac{\|u\|^2}{\|u\|^2_2} \label{1.8}
	\end{equation}

	\subsection{Assumptions and main results} In order to study equation \eqref{1.1} by variational method, we need to take some assumptions on $f$. Assume that the nonlinearity $f$ satisfies the following hypotheses:
	\begin{align} 
		\label{H1} \tag{H1} 
		&f\in C(\R), f(s)=0 \text{ for all } s\leq 0, f \text{ has critical exponential growth defined in \eqref{ceg}. }  \\
		&\text{  Moreover, } f(s)= o(s^{\frac{2 - \mu}{2}}). \nonumber\\
		&\text{ there exists }  K > 1 \text{ such that }  0< K F(s) \leq f(s) s , \text{ for all } s >0. \label{AR} \tag{H2}\\
		&\text{ there exist }  s_0 > 0 ,  M_0 > 0, \text{ and }  v \in (0, 1] \text{ such that }  0 < s^v F(s) \leq M_0 f(s),\text{ for all }\label{d} \tag{H3}\\ 
		&  s \geq s_0 .\nonumber\\
		&\liminf_{s\rightarrow + \infty} \frac{F(s)}{e^{\alpha_0s^2}} := \beta_0 >0. \tag{H4} \label{H4} 
	\end{align}
	Under the assumption \eqref{H1} on $f$, for any $\varepsilon>0$, $q > 1$ and for fixed $\alpha > \alpha_0$, there exists a constant $C=C(\alpha, q, \varepsilon)>0$ such that
	\begin{equation}
		\label{3.1}
		|f(s)| \leq \varepsilon |s|^{\frac{2 - \mu}{2}} + C|s|^{q-1}e^{\alpha s^2}  \quad \forall s \in \mathbb{R}, 
	\end{equation}
	and there exist $\varepsilon_1>0$ and $C_1 >0$ satisfying
	\begin{equation}
		\label{3.2}
		|F(s)| \leq \varepsilon_1 |s|^{\frac{4 - \mu}{2}} + C_1|s|^q e^{\alpha s^2}  \quad \forall s \in \mathbb{R}.
	\end{equation}
	Assumption \eqref{AR} is required in order to prove the Palais-Smale sequence is bounded. Whereas the assumptions \eqref{d} and \eqref{H4} are used to estimate the minimax level of the functional associated with our problem. Moreover, using these two conditions, we do not need the control on $\beta_0$.
	
	The objective of this paper is to establish the existence of a nontrivial solution for \eqref{1.1}. Based on the position of the parameter $\lambda$ relative to the eigenvalues of $(-\Delta, \h)$ with Dirichlet boundary conditions, we present our findings through the following two main theorems. The first theorem addresses the situation when the parameter $\lambda$ lies in the interval $(0,\lambda_1)$. In this case, the classical mountain-pass theorem guarantees the existence of a critical point of the energy functional associated to the problem \eqref{1.1}. Due to the presence of weight function $Q(|x|)$,  we need to work in the function space $\hr$. As we discussed earlier, to prove the solutions of \eqref{1.1} are in $\h$, we need a condition on $\Omega$, i. e. throughout this paper we assume $\Omega$  is $\mathcal{O}(2)$-invariant, where $\mathcal{O}(2)$ is the group of all orthogonal matrices on $\R^2$.
	Our first result can be stated as follows:
	\begin{thm} 
		\label{thm1.1}
		Assume $\lambda \in (0,\lambda_1)$, \eqref{H1}-\eqref{H4} and \eqref{Q} holds. Then, the problem \eqref{1.1} has a nontrivial solution.
	\end{thm}	
	Next, we deal with the case when $\lambda$ lies between two consecutive eigenvalues of $(-\Delta, \h)$, i.e. $\lambda\in(\lambda_k, \lambda_{k+1})$ for $k\geq 1$. In this situation, the classical mountain-pass theorem is not applicable. Instead, we use a milder version of the mountain-pass theorem to prove existence, specifically, the generalized mountain-pass theorem, also known as the linking theorem.
	\begin{thm}
		\label{thm1.2}
		Assume $\lambda \in (\lambda_k, \lambda_{k+1})$, \eqref{H1}-\eqref{H4} and \eqref{Q} holds. Then the problem \eqref{1.1} has a nontrivial solution.
	\end{thm}
	We list some of the contributions of this paper in the literature as follows:
	\begin{enumerate}
		\item[(i)] The non-resonant case, i.e., when $\lambda$ does not equal to any of the eigenvalues of the operator $(-\Delta, \h)$ is not explored yet, whereas the resonant case with local nonlinearity was studied in \cite{Bruno} under the assumption $Q(|x|)\equiv1$. 
		\item[(ii)] If $\lambda>\lambda_1$, we cannot apply the classical mountain-pass theorem to prove existence. We use a generalized version of the mountain-pass theorem in this case. Additionally, proving that the minimax level lies in a certain range is one of the challenges that involves a different approach from the resonant case.
		\item[(iii)] Shen, R\u adulescu and Yang \cite{Shen} proved a nontrivial mountain-pass solution and bound state solution in $H^1(\R^2)$ for Choquard equations of type \eqref{1.1.2} with a weight $Q$ and a potential $V$ decay to zero at infinity. As compared to this weight, we have considered less restrictive conditions on \eqref{Q} adding the linear perturbation.
		\item[(iv)] {Our weight assumptions are general and contain the weights of \cite{Alves1, Shen} as a particular case. Specifically, if we take $Q(|x|)= \frac{1}{|x|^\beta}$ with $\mu+2\beta<2$, then the problem \eqref{1.1} becomes similar to the Stein-Weiss problem as considered in \cite{Alves1, Shen}}.
		\item[(v)] We establish a version of the Sobolev embedding result and Trudinger-Moser inequality related to our problem, which is useful while studying a prototype of the weighted problem defined in \eqref{1.1} 
	\end{enumerate}
	
	The outline of this paper is given as follows:
	In Section \ref{sec2}, we present some preliminary results that will be useful later. In Section \ref{sec3}, we verify certain properties of the Palais-Smale sequences related to the functional. In Section \ref{sec4}, we verify the geometric conditions for the case $0<\lambda<\lambda_1$. Additionally, we provide more detailed information about the minimax level derived from the mountain-pass theorem, and we prove Theorem \ref{1.1}. In Section \ref{sec5}, we consider the case $\lambda \in (\lambda_k, \lambda_{k+1})$.  This section is dedicated to proving Theorem \ref{thm1.2} using linking geometry. 
	\section{Preliminaries}
	\label{sec2}
	In this section, we primarily concentrate on some foundational results that are crucial in this paper.
	\subsection{A version of weighted Sobolev embedding}
	We begin by introducing some initial lemmas that are necessary for proving weighted Sobolev embedding. The following Radial lemma is due to Strauss \cite{Strauss}.
	\begin{lem}
		\label{lem1}
		There exists $R_1>0$ and $C>0$ such that for all $u\in H^1_{0, rad} (\Omega)$
		\begin{equation*}
			|u(x)| \leq C \|u\| |x|^{-\1}, \ \forall \ |x|\geq R_1.
		\end{equation*}  
	\end{lem}
	The proof of the following lemma is given in \cite{Su}.
	\begin{lem}
		\label{lem2}
		\begin{enumerate}
			\item[(i)] Suppose \eqref{Q} holds. For $1\leq p \leq \infty$ and for any $0<r<R<\infty$ with $R>>1$, the embedding $ H^1_{0, rad}(B_R \setminus B_r) \hookrightarrow L^p(B_R\setminus B_r)$ is compact. \item[(ii)] For each open ball $B_R \subset \R^2$, the space $H^1_{0, rad}(B_R)$ is continuously embedded in $H^1_0(B_R)$. In particular, $H^1_{0, rad}(B_R)$ is continuously embedded in $L^p(B_R)$, for $ 1 \leq p <\infty$.
		\end{enumerate}
	\end{lem}
	Next, $1 \leq p < \infty$ we define weighted Lebesgue spaces as
	\[ L_Q^p(\Omega) := \{ u: \overline{\Omega} \rightarrow \R \text{ is measurable and } \int\limits_\Omega Q(|x|) |u|^p dx < \infty\}\]
	endowed with the norm $\Bigg( \int\limits_\Omega Q(|x|)|u|^p dx\Bigg)^{\frac{1}{p}}$.
	Finally, we provide a version of the weighted Sobolev embedding result that is suitable for our problem. The following proof is inspired to \cite{Albuquerque}. 
	\begin{lem}
		\label{lem3}
		Suppose that \eqref{Q} holds true, for all $\frac{4-\mu}{4} \leq p <\infty$, the embedding $\hr \hookrightarrow L^{p_\mu}_{Q_\mu} := \left\{ u: \overline{\Omega} \rightarrow \R \text{ is Lebesgue measurable } \Big| \int\limits_\Omega Q^{\frac{4}{4-\mu}}(|x|) |u|^{\frac{4p}{4-\mu}}dx < \infty \right\}$ is compact.
	\end{lem}
	\begin{proof}
		We prove the lemma in two steps. First, we establish that the embedding from $\hr$ to $L^{p_\mu}_{Q_\mu}$ is continuous. In the second step, we show that this embedding is compact.
		\begin{step}
			Suppose that $0<r<R$. Choosing $r$ small enough and $R\geq R_1$, where $R_1$ is defined in Lemma \ref{lem1}. Then, we claim that for all $u\in \hr$ and $\frac{4-\mu}{4} \leq p <\infty$, there is a constant $C>0$ such that
			\begin{equation}
				\label{2.1}
				\int\limits_\Omega Q^{\frac{4}{4-\mu}}(|x|) |u|^{\frac{4p}{4-\mu}}dx \leq C \Bigg(R^{-\frac{2p}{4-\mu}}\|u\|^{\frac{4p}{4-\mu}}+ r^{\frac{4b_0}{4-\mu}+\frac{2}{v}}\|u\|^{\frac{4p}{4-\mu}} + \int\limits_{\Omega \cap \{B_R\setminus B_r\}} |u|^{\frac{4p}{4-\mu}}dx \Bigg),
			\end{equation}
			where $v>1$ and $b_0v > -\frac{4-\mu}{2}$. 
			
			If \eqref{2.1} holds, then by Lemma \ref{lem2}, we have
			$\int\limits_{\Omega \cap \{B_R\setminus B_r\}} |u|^{\frac{4p}{4-\mu}}dx \leq C_p \|u\|^{\frac{4p}{4-\mu}}.$
			Therefore, we fix $R$  and $r$ in \eqref{2.1}, such that there is a constant $\overline{C}>0$ satisfying
			\begin{equation*}
				\int\limits_\Omega Q^{\frac{4}{4-\mu}}(|x|) |u|^{\frac{4p}{4-\mu}}dx \leq \overline{C} \|u\|^{\frac{4p}{4-\mu}},
			\end{equation*}
			which further implies the continuity of the embedding. Hence, to prove continuous embedding, it is sufficient to verify \eqref{2.1}.
			
			From the assumption \eqref{Q}, we have, $Q(|x|) \leq C_0 |x|^b$, for all $|x| \geq R$. Using the fact that $|x|^b$ is bounded in $\Omega \cap B_R^c$, for any $b\in \R$, we have
			\begin{align*}
				\int\limits_{\Omega \cap B_R^c } Q^{\frac{4}{4-\mu}}(|x|) |u|^{\frac{4p}{4-\mu}}dx &\leq C_0\int\limits_{\Omega \cap B_R^c } |x|^{\frac{4b}{4-\mu}} |u|^{\frac{4p}{4-\mu}}dx 
				\leq C_1 \int\limits_{\Omega \cap B_R^c }|u|^{\frac{4p}{4-\mu}}dx,
			\end{align*}
			this together with Lemma \ref{lem1}, we obtain
			\begin{align}
				\label{2.2}
				\int\limits_{\Omega \cap B_R^c } Q^{\frac{4}{4-\mu}}(|x|) |u|^{\frac{4p}{4-\mu}}dx\leq C_2 R^{-\frac{2p}{4-\mu}}\|u\|^{\frac{4p}{4-\mu}}.
			\end{align}
			Again by hypothesis \eqref{Q}, there exists a constant $C_3>0$ satisfying
			\begin{equation*}
				Q(|x|) \leq C_3 |x|^{b_0} , \text{ for all } 0<|x|<r_0.
			\end{equation*}
			Now, we estimate the integral on $\Omega \cap B_r$. Applying H\"older's inequality with $b_0 v > -\frac{4-\mu}{2}$, we obtain 
			\begin{align}
				\int\limits_{\Omega \cap B_r } Q^{\frac{4}{4-\mu}}(|x|) |u|^{\frac{4p}{4-\mu}}dx &\leq C_3\int\limits_{\Omega \cap B_r } |x|^{\frac{4b_0}{4-\mu}} |u|^{\frac{4p}{4-\mu}}dx \leq C_3 \Bigg(\int\limits_{\Omega \cap B_r } |x|^{\frac{4b_0v}{4-\mu}}dx\Bigg)^{\frac{1}{v}} \Bigg( \int\limits_{\Omega \cap B_r } |u|^{\frac{4pv'}{4-\mu}}dx \Bigg)^{\frac{1}{v'}}. \nonumber
			\end{align}
			Using Lemma \ref{lem2}, we get
			\begin{align}
				\label{2.3}
				\int\limits_{\Omega \cap B_r } Q^{\frac{4}{4-\mu}}(|x|) |u|^{\frac{4p}{4-\mu}}dx
				\leq C_4 r^{\frac{4b_0}{4-\mu}+ \frac{2}{v}} \|u\|^{\frac{4p}{4-\mu}}.
			\end{align}
			Consider the integral on an annular region $\Omega \cap \{B_R \setminus B_r\}$. Because $Q(|x|)$ is continuous on $(0,\infty)$, we can define $M_{r,R} = \displaystyle\max_{r\leq t \leq R} Q(t)$. Using this, we have
			\begin{align}
				\label{2.4}
				\int\limits_{\Omega \cap \{B_R \setminus B_r\}} Q^{\frac{4}{4-\mu}}(|x|) |u|^{\frac{4p}{4-\mu}}dx &\leq M_{r,R}^{\frac{4}{4-\mu}} \int\limits_{\Omega \cap \{B_R \setminus B_r\}} |u|^{\frac{4p}{4-\mu}}dx.
			\end{align}
			Adding \eqref{2.2}, \eqref{2.3} and \eqref{2.4}, we can find a constant $C>0$ such that the estimate in \eqref{2.1} holds. This proves the continuity of the embedding.
		\end{step}
		\begin{step2}
			We assume that $\{u_n\}$ be a sequence in $\hr$ such that $\|u_n\|\leq C$, for some $C>0$. This implies that there exists $u_0 \in \hr$ such that $u_n \rightharpoonup u_0$ in $\hr$. To prove compactness, we need to show that, up to a subsequence, $u_n \rightarrow u_0$ in $L^{p_\mu}_{Q_\mu}$ with $\frac{4-\mu}{4} \leq p <\infty$,
			i.e.
			\[ \lim_{n\rightarrow \infty} \int\limits_{\Omega} Q^{\frac{4}{4-\mu}}(|x|) |u_n-u_0|^{\frac{4p}{4-\mu}}dx =0.\]
			For a given $\varepsilon>0$, choosing $R>0$ sufficiently large and $r>0$ sufficiently small such that the following holds
			\[  r^{\frac{4b_0}{4-\mu}+\frac{2}{v}} < \varepsilon \text{ and } R^{-\frac{2p}{4-\mu}} < \varepsilon.\]
			Since  $H^1_{0, rad}(B_R \setminus B_r)$ is compactly embedded in $L^p(B_R \setminus B_r)$, for $1\leq p <\infty$, it follows that 
			\[ \lim_{n \rightarrow \infty}\int\limits_{\Omega \cap \{B_R\setminus B_r\}} |u_n - u_0|^{\frac{4p}{4-\mu}}dx=0.\]
			Now, as $u_n \rightharpoonup u_0$ in $\hr$, there is a constant $C_5>0$ satisfying $\|u_n-u_0\| \leq C_5$, for all $n$. Therefore, by the above relations, we get
			\begin{align*}
				\int\limits_{\Omega} Q^{\frac{4}{4-\mu}}(|x|) |u_n-u|^{\frac{4p}{4-\mu}}dx
				&\leq C ( 2\varepsilon C_5^{\frac{4p}{4-\mu}} + o_n(1)).
			\end{align*}
			This completes the proof by taking $n\to \infty$.
		\end{step2}
	\end{proof}
	\subsection{A version of weighted Trudinger-Moser inequality}
	As discussed earlier, due to the presence of a weight function $Q$, we  prove the following version of Trudinger-Moser inequality, inspired from \cite{Albuquerque}, that will be an important tool for our results. The statement follows as
	\begin{lem} 
		\label{lem2.2}
		Suppose \eqref{Q} holds. If $u\in \hr$ and $\alpha>0$, then $Q^{\frac{4}{4-\mu}}(|x|)e^{\alpha u^2} \in L^1(\Omega)$. Moreover,
		\begin{equation*}
			\sup_{u\in \hr, \|u\| \leq 1} \int\limits_\Omega Q^{\frac{4}{4-\mu}}(|x|)e^{\alpha u^2}dx < +\infty, \ \text{ for } \ \alpha \leq 4\pi \Big( 1+ \frac{2b_0}{4-\mu}\Big).
		\end{equation*}
	\end{lem} 
	\begin{proof}
		Fix $R>R_1>0$, where $R_1$ is given in Lemma \ref{lem1}. From the assumption \eqref{Q}, we can find $C_0>0$ such that
		\begin{equation}
			\label{2.7}
			Q(|x|) \leq C_0|x|^b, \ \forall \ |x| \geq R, \text{ and } \ Q(|x|) \leq C_0|x|^{b_0}, \ \forall \ 0<|x| \leq R.
		\end{equation}
		Using \eqref{2.7}, then we have
		\begin{align}
			\int\limits_{\Omega \cap B_R^c} Q^{\frac{4}{4-\mu}}(|x|)e^{\alpha u^2}dx = \int\limits_{\Omega \cap B_R^c} Q^{\frac{4}{4-\mu}}(|x|) \sum_{j=0}^{\infty} \frac{\alpha^j u^{2j}}{j!}dx \leq \sum_{j=0}^{\infty} \frac{\alpha^j }{j!}\int\limits_{\Omega \cap B_R^c} |x|^{\frac{4b}{4-\mu}} u^{2j}dx. \nonumber
		\end{align}
		Since $|x|^b$ is bounded on $\Omega \cap B_R^c$, for any $b\in \R$ and applying Lemma \ref{lem1}, we have
		\begin{align}
			\int\limits_{\Omega \cap B_R^c} Q^{\frac{4}{4-\mu}}(|x|)e^{\alpha u^2}dx\leq C_1\sum_{j=0}^{\infty} \frac{\alpha^j }{j!}\int\limits_{\Omega \cap B_R^c} u^{2j}dx \leq C_1\sum_{j=0}^{\infty} \frac{\alpha^j (C\|u\|)^{2j} }{j!}\int\limits_{\Omega \cap B_R^c} |x|^{-j}dx \leq C_2 e^{\alpha C^2\|u\|^2}. \nonumber
		\end{align}
		It follows that, for any $u\in \hr$, we get $Q^{\frac{4}{4-\mu}}(|x|)e^{\alpha u^2} \in$ $L^1(\Omega \cap B_R^c)$. Additionally, we can say that
		\begin{equation}
			\label{2.8}
			\sup_{u\in \hr, \|u\| \leq 1} \int\limits_{\Omega \cap B_R^c} Q^{\frac{4}{4-\mu}}(|x|)e^{\alpha u^2}dx < +\infty, \ \text{ for all } \ \alpha>0.
		\end{equation}
		Now, we estimate the Trudinger-Moser inequality on a ball $B_R$.
		To do this, we need to consider the following two cases:
		\begin{case}
			If $-\frac{4-\mu}{2} <b_0 \leq 0$, assumption \eqref{2.7} and the singular Trudinger-Moser inequality by \cite{Sandeep} implies that
			\begin{align}
				\int\limits_{\Omega \cap B_R} Q^{\frac{4}{4-\mu}}(|x|)e^{\alpha u^2}dx & \leq C_0\int\limits_{\Omega \cap B_R} |x|^{\frac{-4b_0}{4-\mu}}e^{\alpha u^2}dx < \infty, \ \forall \ u\in \hr. \nonumber
			\end{align}
			Moreover, from \eqref{2.5}
			\begin{equation*}
				\sup_{u\in H^1_{0,rad}(\Omega \cap B_R), \|u\|\leq 1} \int\limits_{B_R} |x|^{\beta}  e^{\alpha u^2} dx < \infty, \text{ if and only if }  \alpha \leq 4\pi \Big( 1+ \frac{2b_0}{4-\mu}\Big).
			\end{equation*}
		\end{case}
		\begin{case2}
			If $b_0>0$, by \eqref{2.7} and the classical Trudinger-Moser inequality, we have
			\begin{align*}
				\int\limits_{\Omega \cap B_R} Q^{\frac{4}{4-\mu}}(|x|)e^{\alpha u^2}dx & \leq C_0\int\limits_{\Omega \cap B_R} |x|^{\frac{4b_0}{4-\mu}} e^{\alpha u^2}dx \\
				& \leq C_0 R^{\frac{4b_0}{4-\mu}+2}\int\limits_{\Omega \cap B_1} |x|^{\frac{4b_0}{4-\mu}} e^{\alpha v^2}dx <\infty, \ \forall \ v\in H^1_{0,rad}(\Omega \cap B_1),
			\end{align*}
			where we have used change of variables as $v(x)=u(Rx)$, $x\in B_1$ and also $\|v\|_{L^2(B_1)}=\|u\|_{L^2(B_R)}$. Hence, we can apply a version of the Trudinger-Moser inequality given in \eqref{2.6}. Then, the above integral is finite for all $u\in H^1_{0,rad}(\Omega \cap B_R)$. Moreover, using \eqref{2.6}, we get
			\begin{equation*} 
				\sup_{u\in H^1_{0,rad}(\Omega \cap B_R), \|u\|\leq 1} \int\limits_{\Omega \cap B_R} |x|^{\frac{4b_0}{4-\mu}}  e^{\alpha u^2} dx < \infty, \ \forall \ \alpha \leq 4\pi\Big(1+\frac{2b_0}{4-\mu}\Big).
			\end{equation*}
		\end{case2}
		Combining both the cases and from \eqref{2.8}, we can conclude that the result holds.
	\end{proof}
	\begin{rem}
		\label{rem1}
		It is easy to show that if $u\in \h$, then $Q^{\frac{4}{4-\mu}}(|x|)e^{\alpha u^2} \in L^1(\Omega)$, for any $\alpha>0$. To see this, we consider
		\begin{align}
			\label{2.9}
			\int\limits_\Omega Q^{\frac{4}{4-\mu}}(|x|)e^{\alpha u^2}\ dx \leq C_0\int\limits_{\Omega \cap B_R} |x|^{\frac{4b_0}{4-\mu}}e^{\alpha u^2}\ dx + C_0\int\limits_{\Omega \cap B_R^c} |x|^{\frac{4b}{4-\mu}}e^{\alpha u^2}\ dx.
		\end{align}
		The second integral is finite by the classical Trudinger-Moser inequality together with the fact that $|x|^{\frac{4b}{4-\mu}}$ is bounded on $\Omega \cap B_R^c$. Since $b_0>-\frac{4-\mu}{2}$, choose $t>1$ such that $b_0 t>-\frac{4-\mu}{2}$. Thus, by H\"older's inequality, we have
		\begin{align*}
			\int\limits_{\Omega \cap B_R} |x|^{\frac{4b_0}{4-\mu}}e^{\alpha u^2}\ dx \leq \Bigg( \int\limits_{\Omega \cap B_R} |x|^{\frac{4b_0t}{4-\mu}}\ dx\Bigg)^{\frac{1}{t}} \Bigg( \int\limits_{\Omega \cap B_R} e^{\frac{\alpha t}{t-1} u^2} \ dx\Bigg)^{\frac{t-1}{t}} < \infty.
		\end{align*}
	\end{rem}
	The next lemma is an improvement of the classical Trudinger-Moser inequality. This result is introduced by Lions \cite{Lion}, and we see the following version:
	\begin{lem}
		\label{lem2.3}
		Let $\{u_n\} \subset H^1_{0,rad}(\Omega)$ be a sequence satisfying $\|u_n\|= 1$ for all $n\in \mathbb{N}$ and $u_n \rightharpoonup u $ in $H^1_{0,rad}(\Omega)$ and $0< \|u\| <1$, then for $ 0< p<\frac{4\pi}{1- \|u\|^2} \Big(1+ \frac{2b_0}{4-\mu}\Big)$, we have
		\[\sup_{n \in \mathbb{N}} \int\limits_{\Omega} Q(|x|)e^{ p |u_n|^2}  \,dx < +\infty. \quad\]
	\end{lem}
	The proof can be done by using a similar idea as in \cite{O}.

	\section{Variational framework}
	\label{sec3}
	
	First, we have to show that the functional $J$ introduced in \eqref{J} is well-defined in $\h$. For this, we need to prove the nonlocal term is well-defined. This can be done by using the Proposition \ref{Hardy}.
	
	Given $u \in \h$, taking $s=r= \frac{4}{4-\mu}$ in Proposition \ref{Hardy}, we have
	\begin{align}
		\left| \int\limits_\Omega \Bigg( \int\limits_\Omega \frac{Q(|y|)F(u(y))}{|x-y|^\mu} dy\Bigg) Q(|x|)F(u(x))dx \right| 
		\leq C_\mu\Bigg( \int\limits_{\Omega} Q^{\frac{4}{4-\mu}}(|x|)|F(u(x))|^{\frac{4}{4-\mu}} dx \Bigg)^{\frac{4-\mu}{2}}. \nonumber
	\end{align}
	We use \eqref{3.2} with $\alpha> \alpha_0$ and $q \geq 1$, then
	\begin{align}
		\int\limits_{\Omega} Q^{\frac{4}{4-\mu}}(|x|)|F(u(x))|^{\frac{4}{4-\mu}} dx & \leq C\int\limits_{\Omega} Q^{\frac{4}{4-\mu}}(|x|)|u|^2 dx + C_1\int\limits_\Omega Q^{\frac{4}{4-\mu}}(|x|)|u|^{\frac{4q}{4-\mu}} e^{\frac{4 \alpha}{4-\mu} |u|^2} dx. \nonumber
	\end{align}
	Using Lemma \ref{lem3} with $p=\frac{4-\mu}{2}$ to obtain $\int\limits_{\Omega} Q^{\frac{4}{4-\mu}}(|x|)|u|^2 dx \leq C_q\|u\|^2$.
	By H\"older's inequality and again using Lemma \ref{lem3} with $p=qv$, 
	\begin{align} 
		\int\limits_{\Omega} Q^{\frac{4}{4-\mu}}(|x|)|u|^{\frac{4q}{4-\mu}} e^{\frac{4\alpha |u|^2}{4-\mu}} dx  &\leq \Bigg(\int\limits_{\Omega} Q^{\frac{4}{4-\mu}}(|x|)|u|^{\frac{4qv}{4-\mu}} dx\Bigg)^{\frac{1}{v}} \Bigg( \int\limits_\Omega Q^{\frac{4}{4-\mu}}(|x|)e^{\frac{4\alpha v'|u|^2 }{4-\mu}} dx\Bigg)^{\frac{1}{v'}} \nonumber\\
		& \leq C_{qv}^{\1} \|u\|^{\frac{4q}{4-\mu}} \Bigg( \int\limits_\Omega Q^{\frac{4}{4-\mu}}(|x|)e^{\frac{4\alpha v'|u|^2}{4-\mu}}dx\Bigg)^{\frac{1}{v'}},\nonumber
	\end{align}
	where we have used H\"older's inequality with $v>1$, $\frac{1}{v} + \frac{1}{v'} = 1$. Combining the above relations with Lemma \ref{lem2.2} and Remark \ref{rem1}, we get 
	\begin{align}
		&\left| \int\limits_\Omega \Bigg( \int\limits_\Omega \frac{Q(|y|)F(u(y))}{|x-y|^\mu} dy\Bigg) Q(|x|)F(u(x))dx \right| \nonumber\\
		& \leq C_\mu \Bigg( C C_q \|u\|^2+ C_1 C_{qv}^{\1} \|u\|^{\frac{4q}{4-\mu}} \Bigg( \int\limits_\Omega Q^{\frac{4}{4-\mu}}(|x|)e^{\frac{4\alpha v'|u|^2}{4-\mu}}dx\Bigg)^{\frac{1}{v'}} \Bigg)^{\frac{4-\mu}{2}} \nonumber \\
		& \leq C \|u\|^{4-\mu} + C \|u\|^{2q} \Bigg( \int\limits_\Omega Q^{\frac{4}{4-\mu}}(|x|)e^{\frac{4\alpha v'|u|^2}{4-\mu}}dx\Bigg)^{\frac{4-\mu}{2v'}} <+\infty \nonumber.
	\end{align}
	Hence, the functional $J$ is well-defined on $\h$. Using the standard arguments, we can say that $J$ is $C^1$ and  for all $v \in H^1_{0}(\Omega)$, its derivative is given by
	\[   J'(u)v= \int\limits_\Omega \nabla u \cdot \nabla v\ dx- \lambda \int\limits_\Omega u \cdot v\ dx -\int\limits_\Omega \Bigg( \int\limits_\Omega \frac{Q(|y|)F(u(y))}{|x-y|^{\mu}}dy \Bigg) Q(|x|)f(u(x))v\ dx.\]
	Clearly, by the definition of weak solution, if $u \in \h$ is a critical point of $J$, then it is a weak solution of \eqref{1.1}. As discussed earlier, because of the weight function $Q$, which can be singular at zero, we have used a version of Trudinger-Moser inequality in the radial space $\hr$ for an improved range of $\alpha$. From now on, we restrict ourselves to $\hr$ as function space rather than $\h$.
	
	Now, we recall the definition of the Palais-Smale sequence.
	\begin{definition}  We say that a sequence $\{u_n\} \subset H^1_{0,rad}(\Omega)$ is a Palais-Smale sequence for $J$ at a level $c$ if $J(u_n) \rightarrow c$ and $J'(u_n) \rightarrow 0$ as $n \rightarrow \infty$. Moreover, we say that $J$ satisfies the Palais-Smale condition at level $c$ if every Palais-Smale sequence has a convergent subsequence in $H^1_{0,rad}(\Omega)$.
	\end{definition}
	The following lemma is very crucial for our main results. In the proof, we have used assumption \eqref{AR}, also known as the Ambrosetti-Rabonowitz condition.
	\begin{lem}
		\label{lem3.1}
		Assume \eqref{H1}, \eqref{AR} holds and if $\{ u_n\} \subset \hr $ be a Palais-Smale sequence of $J$, then $\{u_n\}$ is a bounded sequence in $\hr$.
	\end{lem}
	\begin{proof}
		Let $\{u_n\}$ be a Palais-Smale sequence of $J$ at a level $c$ in $\hr$, then by definition 
		\begin{equation}
			\1 \int\limits_\Omega |\nabla u_n|^2 dx -\frac{\lambda}{2}  \int\limits_\Omega |u_n|^2 dx - \1 \int\limits_\Omega \Bigg( \int\limits_\Omega \frac{Q(|y|)F(u_n(y))}{|x-y|^\mu} dy \Bigg) Q(|x|)F(u_n(x)) dx \rightarrow c  \label{3.4}
		\end{equation}  as $n \rightarrow \infty$, and for all $v \in \hr$,
		\begin{equation}
			\left| \int\limits_\Omega \nabla u_n \cdot \nabla v dx -\lambda\int\limits_\Omega u_n \cdot v dx -\int\limits_\Omega \Bigg( \int\limits_\Omega \frac{Q(|y|)F(u_n(y))}{|x-y|^\mu} dy\Bigg) Q(|x|)f(u_n(x))v(x)\ dx \right| \leq \varepsilon_n \|v\|,  \label{3.5}
		\end{equation}
		where $\varepsilon_n \rightarrow 0 \text{ as } n \rightarrow \infty$. Putting $v=u_n$ in \eqref{3.5} and using \eqref{3.4}, there exists $C>0$ such that,
		\begin{align}
			&\int\limits_\Omega \Bigg( \int\limits_\Omega \frac{Q(|y|)F(u_n(y))}{|x-y|^\mu} dy\Bigg) Q(|x|)f(u_n(x)) u_n(x)\ dx \leq \int\limits_\Omega |\nabla u_n|^2 dx - \lambda \int\limits_\Omega |u_n|^2 dx + \varepsilon_n \|u_n\| \nonumber\\
			& \quad \ \ \leq 2C + \int\limits_\Omega \Bigg( \int\limits_\Omega \frac{Q(|y|)F(u_n(y))}{|x-y|^\mu} dy \Bigg) Q(|x|)F(u_n(x)) dx + \varepsilon_n \|u_n\|. \nonumber\\
			& \quad \ \ \leq 2C + \frac{1}{K} \int\limits_\Omega \Bigg( \int\limits_\Omega \frac{Q(|y|)F(u_n(y))}{|x-y|^\mu} dy\Bigg) Q(|x|)f(u_n(x)) u_n(x) dx + \varepsilon_n \|u_n\| \nonumber,
		\end{align}
		where at the last step, we had applied \eqref{AR}. We can find constants $C_1, C_2>0$ such that
		\begin{equation}
			\int\limits_\Omega \Bigg( \int\limits_\Omega \frac{Q(|y|)F(u_n(y))}{|x-y|^\mu} dy\Bigg) Q(|x|)f(u_n(x)) u_n(x) dx \leq C_1 + C_2 \varepsilon_n \|u_n\|.  \label{3.6}
		\end{equation}
		Now, we consider the following two cases:
		\begin{case}
			When $0< \lambda < \lambda_1$.
		\end{case}	
		Since $\hr \subset \h$, using characterization of $\lambda_1$ given in \eqref{1.5}, we get
		\begin{align*}
			\lambda_1 = \min_{u\in \h \setminus \{0\}} \frac{\|u\|^2}{\|u\|^2_2} \leq \min_{u\in \hr \setminus \{0\}} \frac{\|u\|^2}{\|u\|^2_2}
		\end{align*}
		Therefore, $ \displaystyle\|u\|^2_2\leq \frac{1}{\lambda_1}\|u\|^2$ for all $u\in \hr$.
		Taking $v= u_n$ as test function in \eqref{3.5} together with \eqref{3.6}, we get
		\begin{align*}
			\|u_n\|^2 &\leq \lambda \int\limits_\Omega |u_n|^2 dx + \int\limits_\Omega \Bigg( \int\limits_\Omega \frac{Q(|y|)F(u_n(y))}{|x-y|^\mu} dy\Bigg) Q(|x|)f(u_n(x)) u_n(x)\ dx + \varepsilon_n \|u_n\| \\
			& \leq \frac{\lambda}{\lambda_1} \|u_n\|^2 + C_1 + C_2 \varepsilon_n \|u_n\| + \varepsilon_n\|u_n\|,
		\end{align*}
		since $\lambda < \lambda_1$, it follows that $\{ u_n\}$ is a bounded sequence in $\hr$.
		
		\begin{case2}
			When $\lambda_k < \lambda < \lambda_{k+1}$, for $k \geq 1$.
		\end{case2}
		
		We can decompose the radial space $\hr$ as similar to $\h$ by defining two subspaces, $H_{k,r}(\Omega) = \hr \cap H_k(\Omega)$ and $H^{\perp}_{k,r}(\Omega) = \hr \cap H^\perp_k(\Omega)$  of $\hr$. This implies that for any $u \in \hr$, we can write $u= u^k+ u^\perp$, where $u^k \in H_{k,r}(\Omega)$ and $u^\perp \in H^\perp_{k,r}(\Omega)$. It is easy to verify that 
		\begin{align}
			\int\limits_\Omega \nabla u \cdot \nabla u^k dx - \lambda \int\limits_\Omega u \cdot u^k dx &= \|u^k\|^2 - \lambda \|u^k\|^2_2 ~~~\text{ and } \label{3.7}\\
			\int\limits_\Omega \nabla u \cdot \nabla u^\perp dx - \lambda \int\limits_\Omega u \cdot u^\perp dx &= \|u^\perp\|^2 - \lambda \|u^\perp\|^2_2 \label{3.8}.
		\end{align}
		Using this with \eqref{3.5} and \eqref{3.7}, we have 
		\begin{align}
			-\varepsilon_n\|u_n^k\| &\leq \int\limits_\Omega \nabla u_n \nabla u_n^k dx - \lambda \int\limits_\Omega u_n u_n^k dx - \int\limits_\Omega \Bigg( \int\limits_\Omega \frac{Q(|y|)F(u_n(y))}{|x-y|^\mu} dy\Bigg) Q(|x|)f(u_n(x)) u_n^k(x)\ dx \nonumber\\
			& \leq \|u^k_n\|^2 - \lambda \|u^k_n\|^2_2 + \int\limits_\Omega \Bigg( \int\limits_\Omega \frac{Q(|y|)F(u_n(y))}{|x-y|^\mu} dy\Bigg)Q(|x|) \left|f(u_n(x)) u_n^k(x) \right| dx. \nonumber
		\end{align}
		Since $H_{k,r}(\Omega) \subset H_k(\Omega)$, from the characterization of $\lambda_k$ in \eqref{1.7}, we obtain
		\begin{align*}
			\lambda_k = \max_{u\in H_k(\Omega) \setminus \{0\}} \frac{\|u\|^2}{\|u\|_2^2} \geq \max_{u\in H_{k,r}(\Omega) \setminus \{0\}} \frac{\|u\|^2}{\|u\|_2^2}
		\end{align*} 
		This implies that for all $u\in H_{k,r}(\Omega)$, we have $\displaystyle\|u\|_2^2 \geq \frac{1}{\lambda_k} \|u\|^2$. Thus, using the above two relations, we have
		\begin{align}
			-\varepsilon_n\|u_n^k\|& \leq \Bigg(\frac{\lambda_k -\lambda}{\lambda_k}\Bigg) \|u_n^k\|^2 + \|u_n^k\|_\infty \int\limits_\Omega \Bigg( \int\limits_\Omega \frac{Q(|y|)F(u_n(y))}{|x-y|^\mu} dy\Bigg) Q(|x|)|f(u_n(x))| dx, \label{3.10}
		\end{align}
		where we applied H\"older's inequality. To complete the proof, we need to show that there exist constants $C_3, C_4 >0$ such that
		\begin{equation}
			\int\limits_\Omega \Bigg( \int\limits_\Omega \frac{Q(|y|)F(u_n(y))}{|x-y|^\mu} dy\Bigg) Q(|x|)|f(u_n(x))| dx \leq C_3 + C_4 \varepsilon_n \|u_n\|. \label{3.11}
		\end{equation}
		In the next steps, we prove \eqref{3.11}. For this, we define two subsets of $\Omega$ as
		\[  \Omega_n = \{ x\in \Omega: |u_n(x)| \geq 1\} \text{ and } \Omega'_n = \{ x\in \Omega: |u_n(x)| \leq 1\}, \text{ for some fixed n}.\]
		Since $|u_n(x)| \geq 1$ on $\Omega_n$, we have
		\begin{align}
			\int\limits_{\Omega_n} &\Bigg( \int\limits_\Omega  \frac{Q(|y|)F(u_n(y))}{|x-y|^\mu}dy\Bigg) Q(|x|)\left| f(u_n(x))\right| dx \nonumber\\
			&\quad \leq \int\limits_{\Omega_n}\Bigg( \int\limits_\Omega \frac{Q(|y|)F(u_n(y))}{|x-y|^\mu}dy\Bigg) Q(|x|)\left| f(u_n(x))u_n(x)\right|\ dx\nonumber
		\end{align}
		Combining equation \eqref{3.6} and the above, 
		\begin{align}
			\int\limits_{\Omega_n} \Bigg( \int\limits_\Omega  \frac{Q(|y|)F(u_n(y))}{|x-y|^\mu}dy\Bigg) Q(|x|)\left| f(u_n(x))\right| dx \leq C_1 +C_2 \varepsilon_n \|u_n\|. \label{4.1}
		\end{align}
		Now, we estimate the above integral on the set $\Omega'_n$. From Cauchy-Schwarz inequality \cite{Cauchy}, we have
		\begin{align}
			&\int\limits_{\Omega_n'} \Bigg( \int\limits_\Omega \frac{Q(|y|)F(u_n(y))}{|x-y|^\mu}dy\Bigg) Q(|x|)\left| f(u_n(x))\right| dx \nonumber\\
			&  \leq \left[ \int\limits_{\Omega_n'} \Bigg( \int\limits_{\Omega_n'} \frac{Q(|y|)|f(u_n(y))|}{|x-y|^\mu}dy\Bigg) Q(|x|)\left| f(u_n(x))\right| dx \right]^{\1}\times\nonumber\\
			& \quad \quad\left[\int\limits_{\Omega} \Bigg( \int\limits_\Omega \frac{Q(|y|)F(u_n(y))}{|x-y|^\mu}dy\Bigg) Q(|x|) F(u_n(x)) dx\right]^{\1} \nonumber\\
			&  = I_1^{\1} I_2^{\1}. \nonumber
		\end{align}
		Let us first estimate $I_2$. From the assumption \eqref{AR} and equation \eqref{3.6}, we obtain
		\begin{align}
			I_2 & \leq \frac{1}{K}\int\limits_{\Omega} \Bigg( \int\limits_\Omega \frac{Q(|y|)F(u_n(y))}{|x-y|^\mu}dy\Bigg) Q(|x|) f(u_n(x)) u_n(x)\ dx  \leq \frac{1}{K} (C_1+C_2 \varepsilon_n \|u_n\|). \nonumber
		\end{align}
		
		Now, we estimate the integral $I_1$ by Proposition \ref{Hardy}, we have
		\begin{align}
			I_1  = \int\limits_{\Omega_n'} \Bigg( \int\limits_{\Omega_n'} \frac{Q(|y|)|f(u_n(y))|}{|x-y|^\mu}dy\Bigg) Q(|x|)\left| f(u_n(x))\right| dx  \leq C_\mu \Bigg( \int\limits_{\Omega_n'} Q^{\frac{4}{4-\mu}}(|x|) |f(u_n)|^{\frac{4}{4-\mu}}dx\Bigg)^{\frac{4-\mu}{2}}.\nonumber
		\end{align}
		Since $|u_n(x)| \leq 1$ on the set $\Omega_n'$, it follows that from \eqref{3.1} with $q>1$, we can find a constant $C_0>0$ such that
		\begin{align*}
			|f(u_n)| \leq \varepsilon |u_n|^{(2 - \mu)/2} + C|u_n|^{q-1}e^{\alpha |u_n|^2} \leq C_0. 
		\end{align*}
		Using the above inequality together with the assumption \eqref{Q}, there exists $R>0$ such that we can find a constant $\overline{C}>0$ satisfying,
		\begin{align}
			I_1 &\leq C_\mu C_0^2 \Bigg( \int\limits_{\Omega_n'} Q^{\frac{4}{4-\mu}}(|x|) dx\Bigg)^{\frac{4-\mu}{2}}\leq C \Bigg( \int\limits_{\Omega \cap B_R} |x|^{\frac{4b_0}{4-\mu}}dx \Bigg)^{\frac{4-\mu}{2}} + C\Bigg( \int\limits_{\Omega \cap B_R^c} |x|^{\frac{4b}{4-\mu}}dx\Bigg)^{\frac{4-\mu}{2}}\nonumber\\
			& \leq C R^{4+2b_0-\mu} +C \leq \overline{C}, \nonumber
		\end{align}
		this follows from the fact that $|x|^b$ is bounded on $\Omega \cap B_R^c$ for $b\in \R$ and $b_0 > -\frac{4-\mu}{2}$. 
		
		We conclude by combining all the above relations, there exist $C_3, C_4>0$ such that
		\begin{align}
			\label{4.2}
			\int\limits_{\Omega_n'} \Bigg( \int\limits_\Omega \frac{Q(|y|)F(u_n(y))}{|x-y|^\mu}dy\Bigg) Q(|x|)\left| f(u_n(x))\right| dx \leq C_3 + C_4 \varepsilon_n \|u_n\|.
		\end{align}
		Hence, \eqref{3.11} follows by combining \eqref{4.1} and \eqref{4.2}.
		
		Since $H_{k,r}(\Omega)$ is a finite dimensional subspace, it is well known that in a finite dimensional subspace, all the norms are equivalent. Therefore, we can find a constant $C'>0$ such that $\|u_n^k\|_\infty \leq C'\|u_n^k\|$.  Therefore, by \eqref{3.10} and \eqref{3.11}, we get
		\begin{align}
			\Bigg(\frac{\lambda -\lambda_k}{\lambda_k}\Bigg) \|u_n^k\|^2 &\leq \varepsilon_n\|u_n^k\| + C' \|u_n^k\| \int\limits_\Omega \Bigg( \int\limits_\Omega \frac{Q(|y|)F(u_n(y))}{|x-y|^\mu} dy\Bigg) Q(|x|)|f(u_n(x))| dx \nonumber\\
			& \leq \varepsilon_n\|u_n^k\| + C' \|u_n^k\| \big(C_3 + C_4 \varepsilon_n \|u_n\| \big). \nonumber
		\end{align}
		Since $\lambda > \lambda_k$, we can choose a positive constant $C$ such that
		\begin{equation}
			\|u_n^k\|^2 \leq C \big( \|u_n^k\| + \epsilon_n\|u_n^k\|+ \epsilon_n \|u_n^k\| \|u_n\| \big). \label{3.12}
		\end{equation}
		By the variational characterization of $\lambda_{k+1}$ given in \eqref{1.7} and using $H_{k,r}^\perp(\Omega) \subset H_k^\perp(\Omega)$, we obtain 
		\begin{align*}
			\lambda_{k+1} = \min_{u\in H_k^\perp \setminus \{0\}} \frac{\|u\|^2}{\|u\|^2_2} \leq \min_{u\in H_{r,k}^\perp \setminus \{0\}} \frac{\|u\|^2}{\|u\|^2_2}.
		\end{align*}
		Then for $u\in H_{k,r}^\perp(\Omega)$, we have $\displaystyle\|u\|_2^2 \leq \frac{1}{\lambda_{k+1}} \|u\|^2$. Again, by \eqref{3.5} and \eqref{3.8}   
		\begin{align}
			\varepsilon_n\|u_n^k\| &\geq \int\limits_\Omega \nabla u_n \cdot \nabla u_n^\perp dx - \lambda \int\limits_\Omega u_n \cdot u_n^\perp dx - \int\limits_\Omega \Bigg( \int\limits_\Omega \frac{Q(|y|)F(u_n(y))}{|x-y|^\mu} dy\Bigg) Q(|x|)f(u_n(x)) u_n^\perp(x) dx \nonumber\\
			& = \|u^\perp_n\|^2 - \lambda \|u^\perp_n\|^2_2 - \int\limits_\Omega \Bigg( \int\limits_\Omega \frac{Q(|y|)F(u_n(y))}{|x-y|^\mu} dy\Bigg) Q(|x|)f(u_n(x)) u_n^\perp dx  \nonumber\\
			& \geq \Bigg(\frac{\lambda_{k+1} -\lambda}{\lambda_{k+1}}\Bigg) \|u_n^\perp\|^2 - \int\limits_\Omega \Bigg( \int\limits_\Omega \frac{Q(|y|)F(u_n(y))}{|x-y|^\mu} dy\Bigg) Q(|x|)f(u_n(x)) u_n dx \nonumber\\
			& \quad+ \int\limits_\Omega \Bigg( \int\limits_\Omega \frac{Q(|y|)F(u_n(y))}{|x-y|^\mu} dy\Bigg) Q(|x|)f(u_n(x)) u_n^k dx. \nonumber
		\end{align}
		Using \eqref{3.6}, we have
		\begin{align}
			\varepsilon_n\|u_n^k\| & \geq \Bigg(\frac{\lambda_{k+1} -\lambda}{\lambda_{k+1}}\Bigg) \|u_n^\perp\|^2 - C_1 - C_2 \epsilon_n \|u_n\|\nonumber\\
			&- \int\limits_\Omega \Bigg( \int\limits_\Omega \frac{Q(|y|)F(u_n(y))}{|x-y|^\mu} dy\Bigg)Q(|x|) \left|f(u_n(x))u_n^k\right| dx. \nonumber
		\end{align}
		By H\"older's inequality, we get
		\begin{align}
			\varepsilon_n\|u_n^k\|	& \geq \Bigg(\frac{\lambda_{k+1} -\lambda}{\lambda_{k+1}}\Bigg) \|u_n^\perp\|^2 - C_1 - C_2 \epsilon_n \|u_n\| \nonumber\\
			& \quad- \|u_n^k\|_\infty\int\limits_\Omega \Bigg( \int\limits_\Omega \frac{Q(|y|)F(u_n(y))}{|x-y|^\mu} dy\Bigg) Q(|x|)|f(u_n(x))| dx. \nonumber
		\end{align}
		Since $H_{k,r}(\Omega)$ is a finite dimensional subspace, there exists a constant $C'>0$ such that $\|u_n^k\|_\infty \leq C'\|u_n^k\|$. This together with \eqref{3.11} gives
		\begin{align}
			\Bigg(\frac{\lambda_{k+1} -\lambda}{\lambda_{k+1}}\Bigg) \|u_n^\perp\|^2 &\leq \varepsilon_n\|u_n^\perp\|+ C_1 + C_2 \varepsilon_n \|u_n\| \nonumber\\
			& \quad+ \|u_n^k\|_\infty\int\limits_\Omega \Bigg( \int\limits_\Omega \frac{Q(|y|)F(u_n(y))}{|x-y|^\mu} dy\Bigg) Q(|x|)|f(u_n(x))| dx, \nonumber\\
			& \leq \varepsilon_n\|u_n^\perp\|+ C_1 + C_2 \varepsilon_n \|u_n\| \nonumber\\
			&\quad+ C'\|u_n^k\|\int\limits_\Omega \Bigg( \int\limits_\Omega \frac{Q(|y|)F(u_n(y))}{|x-y|^\mu} dy\Bigg) Q(|x|)|f(u_n(x))| dx, \nonumber\\
			& \leq \varepsilon_n\|u_n^\perp\|+ C_1 + C_2 \varepsilon_n \|u_n\| + C'\|u_n^k\| \big(C_3+ C_4 \varepsilon_n \|u_n\| \big). \nonumber 
		\end{align}
		Thus, we can find a positive constant $C>0$ such that 
		\begin{align}
			\|u_n^\perp\|^2 \leq C\big( 1+\varepsilon_n\|u_n^\perp\| + \varepsilon_n \|u_n\| + \|u_n^k\|+\varepsilon_n \|u_n^k\| \|u_n\|\big). \label{3.13}
		\end{align}
		Adding \eqref{3.12} and \eqref{3.13}, we can find a constant $C>0$ such that
		\begin{align}
			\|u_n\|^2  \leq C( 1+ \|u_n\| + \varepsilon_n \|u_n\|+ \varepsilon_n \|u_n\|^2). \nonumber
		\end{align}
		In this case also the sequence $\{u_n\}$ is bounded in $\hr$. This completes the proof.
	\end{proof}
	\begin{rem}
		\label{rem3} The above result holds even if we consider the sequence $\{u_n\}$ in $\h$.
	\end{rem}
	Since every Palais-Smale sequence $\{u_n\}$ is bounded in $\hr$ and the space $\hr$ is reflexive, therefore, by the Banach-Alaoglu Theorem, $\{u_n\}$ is relatively compact in $\hr$, i.e. up to a subsequence, there exists $u_0 \in \hr$ such that $u_n \rightharpoonup u_0$ in $\hr$. We know that $\hr$ is compactly embedded in $L^p(\Omega)$ for $ 1 \leq p < \infty$, when $\Omega$ is bounded. This implies that $u_n \rightarrow u_0$ in $L^p(\Omega)$. Furthermore, this implies up to a subsequence, $u_n(x) \rightarrow u_0(x)$ a.e. in $\Omega$.
	
	We need to prove the Palais-Smale condition in $\hr$ to prove the existence result. For this, we establish the following lemmas, which are very important tools in proving the Palais-Smale condition.
	\begin{lem} \label{lem3.2}
		Assume \eqref{H1} and \eqref{d} holds. If $\{u_n\}$ is a Palais-Smale sequence in $\hr$, then up to a subsequence, we have
		\begin{align*}
			\lim_{n \rightarrow \infty} &\int\limits_{\Omega} \Bigg( \int\limits_{\Omega} \frac{Q(|y|)F(u_n(y))}{|x-y|^\mu}dy \Bigg) Q(|x|)F(u_n(x))dx\\
			&=
			\int\limits_{\Omega} \Bigg( \int\limits_{\Omega} \frac{Q(|y|)F(u_0(y))}{|x-y|^\mu}dy \Bigg) Q(|x|)F(u_0(x))dx. 
		\end{align*}  
	\end{lem}
	\begin{proof} Since $\{u_n\}$ is a Palais-Smale sequence in $\hr$, then from Lemma \ref{lem3.1}, it is bounded. This together with equation \eqref{3.5} implies that  
		there exists a constant $C_0>0$ such that
		\begin{align}
			\sup_n\int\limits_{\Omega} \Bigg( \int\limits_{\Omega} \frac{Q(|y|)F(u_n(y))}{|x-y|^\mu}dy \Bigg) Q(|x|)f(u_n(x))u_n(x)dx \leq C_0. \label{3.14}
		\end{align}
		Since we have shown that $J$ is well defined, one can have 
		\begin{equation}
			\label{3.15} \Bigg( \int\limits_{\Omega} \frac{Q(|y|)F(u(y))}{|x-y|^\mu}dy \Bigg) Q(|x|)F(u) \in L^1(\Omega), \text{ for all } u \in \hr.
		\end{equation}
		It follows that for a given $\varepsilon>0$, we can choose $M_\varepsilon> \max\left\{ \Big(\frac{C_0 M_0}{\varepsilon}\Big)^{\frac{1}{v+1}}, s_0\right\}$,  	where $M_0, s_0, v$ are defined in \eqref{d} and $C_0$ is given in \eqref{3.14}, satisfying
		\begin{equation}
			\label{3.16}
			\int\limits_{\Omega \cap \{ |u_0|\geq M_\varepsilon\}} \Bigg( \int\limits_{\Omega} \frac{Q(|y|)F(u_0(y))}{|x-y|^\mu}dy \Bigg) Q(|x|)F(u_0(x))dx \leq \varepsilon.
		\end{equation}
		From \eqref{d}, we have
		\begin{align}
			\int\limits_{\Omega \cap \{ |u_n|\geq M_\varepsilon\}} &\Bigg( \int\limits_{\Omega} \frac{Q(|y|)F(u_n(y))}{|x-y|^\mu}dy \Bigg) Q(|x|)F(u_n(x))dx \nonumber\\
			&\leq M_0 \int\limits_{\Omega \cap \{ |u_n|\geq M_\varepsilon\}} \Bigg( \int\limits_{\Omega} \frac{Q(|y|)F(u_n(y))}{|x-y|^\mu}dy \Bigg)  \frac{Q(|x|)|f(u_n(x))|}{|u_n|^v}dx\nonumber\\
			&\leq \frac{M_0}{M_\varepsilon^{v+1}} \int\limits_{\Omega \cap \{ |u_n|\geq M_\varepsilon\}} \Bigg( \int\limits_{\Omega} \frac{Q(|y|)F(u_n(y))}{|x-y|^\mu}dy \Bigg) Q(|x|)f(u_n(x))u_n(x)dx\nonumber
		\end{align}
		Using \eqref{3.14}
		\begin{align}
			\label{3.17}
			\int\limits_{\Omega \cap \{ |u_n|\geq M_\varepsilon\}} \Bigg( \int\limits_{\Omega} \frac{Q(|y|)F(u_n(y))}{|x-y|^\mu}dy \Bigg) Q(|x|)F(u_n(x))dx 
			\leq \frac{M_0 C_0}{M_\varepsilon^{v+1}} <\varepsilon.
		\end{align}
		Then from \eqref{3.16} and \eqref{3.17}, we obtain
		\begin{align}
			\Bigg| \int\limits_{\Omega} &\Bigg( \int\limits_{\Omega} \frac{Q(|y|)F(u_n(y))}{|x-y|^\mu}dy \Bigg) Q(|x|)F(u_n(x))dx -\int\limits_{\Omega} \Bigg( \int\limits_{\Omega} \frac{Q(|y|)F(u_0(y))}{|x-y|^\mu}dy \Bigg) Q(|x|)F(u_0(x))dx\Bigg| \nonumber\\
			& \leq 2\varepsilon +\Bigg| \int\limits_{\Omega \cap \{ |u_n| \leq M_\varepsilon\}} \bigg[\Bigg( \int\limits_{\Omega} \frac{Q(|y|)F(u_n(y))}{|x-y|^\mu}dy \Bigg) Q(|x|)F(u_n(x))\nonumber\\
			&  \quad- \Bigg( \int\limits_{\Omega} \frac{Q(|y|)F(u_0(y))}{|x-y|^\mu}dy \Bigg) Q(|x|)F(u_0(x))\bigg]dx \Bigg|. \nonumber
		\end{align}
		Thus, it is sufficient to show that
		\begin{align}&\displaystyle\int\limits_{\Omega \cap \{ |u_n| \leq M_\varepsilon\}} \Bigg( \int\limits_{\Omega} \frac{Q(|y|)F(u_n(y))}{|x-y|^\mu}dy \Bigg) Q(|x|)F(u_n(x))dx  \nonumber\\
			&\quad \to \int\limits_{\Omega \cap \{|u_0| \leq M_\varepsilon\}} \Bigg( \int\limits_{\Omega} \frac{Q(|y|)F(u_0(y))}{|x-y|^\mu}dy \Bigg) Q(|x|)F(u_0(x))dx \ \text{ as } \  \ n\to \infty. \nonumber
		\end{align}
		By Fubini's theorem, we have 
		\begin{align*}
			\int\limits_{\Omega \cap \{|u_0| \leq M_\varepsilon\}} \Bigg( \int\limits_{\Omega \cap \{|u_0|\geq K_\varepsilon\}} \frac{Q(|y|)F(u_0(y))}{|x-y|^\mu}dy \Bigg) Q(|x|)F(u_0(x))dx =\\
			\quad \int\limits_{\Omega \cap \{|u_0|\geq K_\varepsilon\}} \Bigg( \int\limits_{\Omega \cap \{|u_0| \leq M_\varepsilon\}} \frac{Q(|y|)F(u_0(y))}{|x-y|^\mu}dy \Bigg) Q(|x|)F(u_0(x))dx.
		\end{align*}
		Now, again using \eqref{3.15} and the above relation, for any $\varepsilon>0$, we choose $K_\varepsilon > \max\left\{ \Big(\frac{C_0 M_0}{\varepsilon}\Big)^{\frac{1}{v+1}}, s_0\right\}$ satisfying
		\begin{equation}
			\label{3.18}
			\int\limits_{\Omega \cap \{|u_0| \leq M_\varepsilon\}} \Bigg( \int\limits_{\Omega \cap \{|u_0|\geq K_\varepsilon\}} \frac{Q(|y|)F(u_0(y))}{|x-y|^\mu}dy \Bigg) Q(|x|)F(u_0(x))dx \leq \varepsilon.
		\end{equation} 
		Using \eqref{d} and Fubini's theorem, one can have
		\begin{align}
			\int\limits_{\Omega \cap \{|u_n| \leq M_\varepsilon\}}& \Bigg( \int\limits_{\Omega \cap \{|u_n|\geq K_\varepsilon\}} \frac{Q(|y|)F(u_n(y))}{|x-y|^\mu}dy \Bigg) Q(|x|)F(u_n(x))dx \nonumber\\
			&\quad\leq M_0 \int\limits_{\Omega \cap \{ |u_n|\leq M_\varepsilon\}} \Bigg( \int\limits_{\Omega \cap \{ |u_n|\geq K_\varepsilon\}} \frac{Q(|y|)|f(u_n(y))|}{|u_n|^v|x-y|^\mu}dy \Bigg)  Q(|x|)F(u_n(x))dx\nonumber\\
			&\quad \leq \frac{M_0}{K_\varepsilon^{v+1}} \int\limits_{\Omega \cap \{ |u_n|\leq M_\varepsilon\}} \Bigg( \int\limits_{\Omega \cap \{|u_n|\geq K_\varepsilon\}} \frac{Q(|y|)f(u_n(y))u_n(y)}{|x-y|^\mu}dy \Bigg) Q(|x|)F(u_n(x))dx\nonumber\\
			&\quad = \frac{M_0}{K_\varepsilon^{v+1}} \int\limits_{\Omega \cap \{ |u_n|\geq K_\varepsilon\}} \Bigg( \int\limits_{\Omega \cap \{|u_n|\leq M_\varepsilon\}} \frac{Q(|y|)F(u_n(y))}{|x-y|^\mu}dy \Bigg) Q(|x|)f(u_n(x)) u_n(x)dx.\nonumber
		\end{align}
		From \eqref{3.14}, we have
		\begin{align}
			\label{3.19}
			\int\limits_{\Omega \cap \{|u_n| \leq M_\varepsilon\}} \Bigg( \int\limits_{\Omega \cap \{|u_n|\geq K_\varepsilon\}} \frac{Q(|y|)F(u_n(y))}{|x-y|^\mu}dy \Bigg) Q(|x|)F(u_n(x))dx
			\leq \frac{K_0 C_0}{M_\varepsilon^{v+1}} <\varepsilon.
		\end{align}
		Combining \eqref{3.18} and \eqref{3.19}, we have
		\begin{align*}
			&\Bigg|\int\limits_{\Omega \cap \{ |u_n| \leq M_\varepsilon\}} \Bigg( \int\limits_{\Omega} \frac{Q(|y|)F(u_n(y))}{|x-y|^\mu}dy \Bigg) Q(|x|)F(u_n(x))dx  - \\
			& \quad \int\limits_{\Omega \cap \{|u_0| \leq M_\varepsilon\}} \Bigg( \int\limits_{\Omega} \frac{Q(|y|)F(u_0(y))}{|x-y|^\mu}dy \Bigg) Q(|x|)F(u_0(x))dx\Bigg|\\
			&\leq 2\varepsilon + \Bigg|\int\limits_{\Omega \cap \{ |u_n| \leq M_\varepsilon\}} \Bigg( \int\limits_{\Omega \cap \{ |u_n| \leq K_\varepsilon\}} \frac{Q(|y|)F(u_n(y))}{|x-y|^\mu}dy \Bigg) Q(|x|)F(u_n(x))dx\\
			& \quad  - \int\limits_{\Omega \cap \{|u_0| \leq M_\varepsilon\}} \Bigg( \int\limits_{\Omega \cap \{|u_n| \leq K_\varepsilon\}} \frac{Q(|y|)F(u_0(y))}{|x-y|^\mu}dy \Bigg) Q(|x|)F(u_0(x))dx\Bigg|.
		\end{align*}
		It remains to show that 
		\begin{align*}
			&\int\limits_{\Omega \cap \{ |u_n| \leq M_\varepsilon\}} \Bigg( \int\limits_{\Omega \cap \{ |u_n| \leq K_\varepsilon\}} \frac{Q(|y|)F(u_n(y))}{|x-y|^\mu}dy \Bigg) Q(|x|)F(u_n(x))dx \\ &\quad \to \int\limits_{\Omega \cap \{|u_0| \leq M_\varepsilon\}} \Bigg( \int\limits_{\Omega \cap \{|u_n| \leq K_\varepsilon\}} \frac{Q(|y|)F(u_0(y))}{|x-y|^\mu}dy \Bigg) Q(|x|)F(u_0(x))dx \ \text{ as } \ n \to \infty.
		\end{align*}
		
		Let $\epsilon =1$ and $q=1$ in \eqref{3.1}, then there exists a constant $\overline{s}>0$, for all $|s| \leq \overline{s}$ such that
		\begin{align*}
			|F(s)| &\leq |s|^{(4-\mu)/2} + C|s| e^{\alpha s^2} \nonumber \\
			& \leq |s|^{(4-\mu)/2} + C|s| \sum_{j=0}^\infty \frac{\alpha^js^{2j}}{j!} \nonumber\\
			& \leq |s|^{(4-\mu)/2} \Bigg(1 + C |s|\sum_{j=1}^\infty \frac{\alpha^j|s|^{2j-(4-\mu)/2}}{j!} \Bigg) + C|s|\nonumber
		\end{align*}
		Hence, for all $|s| \leq \overline{s},$ we have
		\begin{align}
			\label{f}
			|F(s)| \leq C_{\overline{s}} |s|^{(4-\mu)/2}.
		\end{align}
		Since $\{u_n\}$ is a Palais-Smale sequence in $\hr$, Lemma \ref{lem3.1} implies that it is bounded. Using this with Remark \ref{rem3}, we have $u_n(x) \to u_0(x)$ a. e. in $\Omega$. The following convergence holds by Brezis-Lieb lemma for the nonlocal term \cite{Moroz} together with Proposition \ref{Hardy}
		\begin{align*}
			&\Bigg( \int\limits_{\Omega} \frac{Q(|y|)|u_n(y)|^{\frac{4-\mu}{2}}}{|x-y|^\mu}dy \Bigg) Q(|x|)|u_n(x)|^{\frac{4-\mu}{2}} \chi_{\Omega} \to \nonumber\\
			&\quad \Bigg( \int\limits_{\Omega} \frac{Q(|y|)|u_0(y)|^{\frac{4-\mu}{2}}}{|x-y|^\mu}dy \Bigg) Q(|x|)|u_0(x)|^{\frac{4-\mu}{2}} \chi_{\Omega} \ \text{ in } L^1(\Omega).
		\end{align*}
		It is easy to show that up to a subsequence, the following holds 
		\begin{align*}
			&\Bigg( \int\limits_{\Omega \cap \{ |u_n| \leq K_\varepsilon\}} \frac{Q(|y|)F(u_n(y))}{|x-y|^\mu}dy \Bigg) Q(|x|)F(u_n(x)) \chi_{\Omega \cap \{ |u_n| \leq M_\varepsilon\}} \to \nonumber \\
			& \quad \Bigg( \int\limits_{\Omega \cap \{|u_0| \leq K_\varepsilon\}} \frac{Q(|y|)F(u_0(y))}{|x-y|^\mu}dy \Bigg) Q(|x|)F(u_0(x)) \chi_{\Omega \cap \{|u_0| \leq M_\varepsilon\}} \ \text{ a.e. in } \Omega. \label{s}
		\end{align*}
		
		From \eqref{f}, there exist constants $C_{M_\varepsilon}, C_{K_\varepsilon}>0$ depending on $M_\varepsilon$ and $K_\varepsilon$ respectively, such that
		\begin{align*}
			& \Bigg( \int\limits_{\Omega \cap \{ |u_n| \leq K_\varepsilon\}} \frac{Q(|y|)F(u_n(y))}{|x-y|^\mu}dy \Bigg) Q(|x|)F(u_n(x)) \chi_{\Omega \cap \{ |u_n| \leq M_\varepsilon\}}\\
			& \quad \leq C_{M_\varepsilon} C_{K_\varepsilon} \Bigg( \int\limits_{\Omega } \frac{Q(|y|) |u_n(y)|^{\frac{4-\mu}{2}}}{|x-y|^\mu}\Bigg) Q(|x|) |u_n(x)|^{\frac{4-\mu}{2}} \\
			& \quad \to C_{M_\varepsilon} C_{K_\varepsilon} \Bigg( \int\limits_{\Omega } \frac{Q(|y|) |u_0(y)|^{\frac{4-\mu}{2}}}{|x-y|^\mu}\Bigg) Q(|x|) |u_0(x)|^{\frac{4-\mu}{2}} \text{ in } L^1(\Omega),
		\end{align*}
		where we have used Proposition \ref{Hardy}. Therefore, using generalized Lebesgue's dominated convergence theorem, the proof is completed.

	\end{proof} 
	\begin{rem}
		\label{rem3.1}
		Lemma \ref{lem3.2} holds even if we take $\{u_n\}$ to be a sequence converges weakly to $u_0$ in $\hr$ and satisfies \eqref{3.14}.
	\end{rem}
	\begin{lem}
		\label{lemma3.4}
		Assume \eqref{H1} and \eqref{d} holds. If $\{u_n\}$ is a Palais-Smale sequence for $J$ in $\hr$. Then, for all $\phi \in C^\infty_{0,rad}(\Omega)$, up to a subsequence, we can conclude that
		\begin{align*}
			&\lim_{n \rightarrow \infty} \int\limits_{\Omega} \Bigg( \int\limits_{\Omega} \frac{Q(|y|)F(u_n(y))}{|x-y|^\mu}dy \Bigg) Q(|x|)f(u_n(x)) \phi(x) dx \\
			& \quad= \int\limits_{\Omega} \Bigg( \int\limits_{\Omega} \frac{Q(|y|) F(u_0(y))}{|x-y|^\mu}dy \Bigg) Q(|x|)f(u_0(x)) \phi(x) dx. 
		\end{align*}  
	\end{lem}
	\begin{proof}
		Since $u_n$ is a Palais-Smale sequence in $\hr$ and from Lemma \ref{lem3.1}, it is bounded. Then, there exists a $u_0\in \hr$ such that up to a subsequence $u_n\rightharpoonup u_0$ in $\hr$, $u_n \to u_0$ in $L^p(\Omega)$ with $1\leq p <\infty$ and $u_n(x)\to u_0(x)$ a.e. in $\Omega$. Without loss of generality, we can assume that $u_n \geq 0$, as $J(|u|) \leq J(u)$ for all $u\in \hr$. 
		For each $n\in \mathbb{N}$, we define a non-negative sequence $\zeta_n$ as
		\begin{equation*}
			\zeta_n(x): = \Bigg( \int\limits_\Omega \frac{Q(|y|)F(u_n(y))}{|x-y|^\mu}dy\Bigg) Q(|x|)f(u_n(x)) \ + \ \lambda u_n(x).
		\end{equation*}
		Now, we estimate the integral  
		\begin{align}
			\int\limits_{\Omega} \zeta_n dx 
			& \leq \int\limits_{\Omega} \Bigg( \int\limits_\Omega \frac{Q(|y|)F(u_n(y))}{|x-y|^\mu}dy\Bigg) Q(|x|)f(u_n(x)) dx \ + \ \lambda\int\limits_{\Omega} u_n(x) \ dx.\nonumber
		\end{align}
		Using H\"older's inequality and \eqref{3.11}, we get a constant $C>0$ satisfying
		\begin{align}
			\int\limits_{\Omega} \zeta_n dx \leq
			C + C \varepsilon_n \|u_n\| \ + \lambda |\Omega|^{\1} \Big( \int\limits_{\Omega} |u_n(x)|^2 \ dx\Big)^{\1}.\nonumber
		\end{align}
		By the embedding $\hr \hookrightarrow L^p(\Omega)$ with $1\leq p < \infty$, for bounded domain $\Omega$, we have $\|u\|_2 \leq C_1 \|u\|$ for all $u\in \hr$. This together with the fact that $\{u_n\}$ is bounded in $\hr$, we obtain a constant $C'>0$, independent of $n$ satisfying 
		\begin{align}
			\int\limits_{\Omega} \zeta_n(x) \ dx \leq C+ \lambda |\Omega|^{\1} \|u_n\| < C'. \nonumber
		\end{align}
		This implies that the sequence $\{\zeta_n\}$ is bounded in $L^1_{loc}(\Omega)$, then up to a subsequence, there exists a Radon measure $\mu$ such that $\zeta_n \rightharpoonup \mu$ in the weak* topology, i.e. 
		\begin{align*}
			\lim_{n\to \infty} \int\limits_\Omega \left[\Bigg( \int\limits_\Omega \frac{Q(|y|)F(u_n(y))}{|x-y|^\mu}dy\Bigg) Q(|x|)f(u_n(x)) \phi + \lambda u_n \phi \right] \ dx= \int\limits_\Omega \phi \ d\mu, \ \ \forall \ \phi \in C_{0,rad}^\infty(\Omega).
		\end{align*}
		Since $\{u_n\}$ is a Palais-Smale sequence, so it satisfies \eqref{3.5}, then we have
		\begin{align*}
			\lim_{n\to \infty} \int\limits_\Omega \nabla u_n \cdot \nabla \phi \ dx  =\int\limits_\Omega \phi \ d\mu.
		\end{align*}
		As $u_n \rightharpoonup u_0$ in $\hr$, this implies that
		$\displaystyle\int\limits_\Omega \nabla u_0 \cdot \nabla \phi \ dx =\int\limits_\Omega \phi \ d\mu.$
		
		Therefore, the Radon measure $\mu$ is absolutely continuous with respect to the Lebesgue measure. Hence, the Radon-Nikodym theorem implies that there exists a function $\zeta \in L^1_{loc}(\Omega)$ such that, for any $\phi \in C_{0,rad}^\infty(\Omega)$
		\begin{align*}
			\int\limits_\Omega \phi \ d\mu = \int\limits_\Omega \phi \zeta \ dx.
		\end{align*}
		Then, we conclude using the above relations that
		\begin{align*}
			&\lim_{n\to \infty} \int\limits_\Omega \left[ \Bigg( \int\limits_\Omega \frac{Q(|y|)F(u_n(y))}{|x-y|^\mu}dy\Bigg) Q(|x|)f(u_n(x)) \phi + \lambda u_n \phi \right]\ dx \\
			&= \int\limits_\Omega \phi \zeta \ dx
			= \int\limits_\Omega \left[\Bigg( \int\limits_\Omega \frac{Q(|y|)F(u_0(y))}{|x-y|^\mu}dy\Bigg) Q(|x|)f(u_0(x)) \phi+ \lambda u_0 \phi\right] \ dx.
		\end{align*}
		Since $u_n \to u_0$ a.e. in $\Omega$, so $\zeta$ can be identified as $\displaystyle\Bigg( \int\limits_\Omega \frac{Q(|y|)F(u_0(y))}{|x-y|^\mu}dy\Bigg)Q(|x|)f(u_0(x)) + \lambda u_0(x)$.
		This implies that
		\begin{align*}
			&\lim_{n\to \infty} \int\limits_\Omega \Bigg( \int\limits_\Omega \frac{Q(|y|)F(u_n(y))}{|x-y|^\mu}dy\Bigg) Q(|x|)f(u_n(x)) \phi\ dx \\
			& \quad=\int\limits_\Omega \Bigg( \int\limits_\Omega \frac{Q(|y|)F(u_0(y))}{|x-y|^\mu}dy\Bigg) Q(|x|)f(u_0(x)) \phi \ dx.
		\end{align*}
		This completes the proof.
	\end{proof}
	
	\section{Mountain-pass case when $0<\lambda<\lambda_1$}
	\label{sec4}
	To prove the existence of a nontrivial solution of equation \eqref{1.1} when $0< \lambda< \lambda_1$, we use the mountain-pass theorem due to A. Ambrosetti and P. Rabinowitz \cite{Ambrosetti}.
	\begin{thm}
		\label{thm4.1}
		Let $J : \mathcal{H} \rightarrow \R$ be a $C^1$ functional on a Banach space $(\mathcal{H}, \|\cdot\|)$ satisfying $J(0)=0.$ Assume that there exist positive numbers $\rho$ and $\delta$ such that
		\begin{enumerate}
			\item[(i)] $J(u)\geq \delta$ for all $u\in \mathcal{H}$ satisfying $\|u\| = \rho$.
			\item[(ii)] There exists $v\in \mathcal{H}$ such that $J(v)<\delta$ for some $v\in \mathcal{H}$ with $\|v\|\geq \rho$.
			\item[(iii)] There exists some $\beta >0$ such that $J$ satisfies the Palais-Smale condition, for all $c\in(0, \beta)$.
		\end{enumerate}
		Consider $\displaystyle\Gamma=\{\gamma\in C([0, 1], \mathcal{H}) : \gamma(0) = 0 \text{ and } \gamma(1) = v\}$ and set $\displaystyle c = \inf_{\gamma \in \Gamma} \max_{t\in [0,1]}J(\gamma(t)) \geq \delta$. Then $c \in (0, \beta)$ is a critical value of the functional $J$.
	\end{thm}
	
	In the following propositions, we show that $J$ satisfies the geometry (i) and (ii).
	\begin{prop}
		\label{prop4.2}
		Assume \eqref{H1}. Then there exists $\delta, \rho>0$ such that $J(u)\geq \delta$, for $u\in \hr$ satisfying $\|u\|= \rho$.
	\end{prop}
	\begin{proof}
		For any $\varepsilon>0$, $q>1$ and $\alpha>\alpha_0$, there exists $C>0$ in \eqref{3.2}, we have
		\begin{align}
			\int\limits_\Omega Q^{\frac{4}{4-\mu}}(|x|)|F(u)|^{\frac{4}{4-\mu}}dx &\leq \int\limits_\Omega Q^{\frac{4}{4-\mu}}(|x|) \Bigg(\varepsilon_1 |u|^{\frac{4-\mu}{2}} + C|u|^q e^{\alpha |u|^2}\Bigg)^{\frac{4}{4-\mu}} dx\nonumber\\
			&\leq C_1 \int\limits_\Omega Q^{\frac{4}{4-\mu}}(|x|) |u|^2 dx +C_2 \int\limits_\Omega Q^{\frac{4}{4-\mu}}(|x|) |u|^{\frac{4q}{4-\mu}} e^{\frac{4\alpha |u|^2}{4-\mu}} dx.\nonumber
		\end{align}
		Now, from Lemma \ref{lem3} with $p= \frac{4-\mu}{2}$, we have
		\begin{align*}
			\int\limits_\Omega Q^{\frac{4}{4-\mu}}(|x|) |u|^2 dx \leq C \|u\|^2.
		\end{align*}
		Using H\"older's inequality with $\frac{1}{v}+\frac{1}{v'}=1$, \ $qv\geq \frac{4-\mu}{2}$ and Lemma \ref{lem3}, we have
		\begin{align*}
			\int\limits_\Omega Q^{\frac{4}{4-\mu}}(|x|) |u|^{\frac{4q}{4-\mu}} e^{\frac{4\alpha |u|^2}{4-\mu}} dx &\leq \Bigg( \int\limits_\Omega Q^{\frac{4}{4-\mu}}(|x|) |u|^{\frac{4qv}{4-\mu}} dx\Bigg)^{\frac{1}{v}} \Bigg( \int\limits_\Omega Q^{\frac{4}{4-\mu}}(|x|) e^{\frac{4\alpha v' |u|^2}{4-\mu}} dx\Bigg)^{\frac{1}{v'}}\\
			&\leq C\|u\|^{\frac{4q}{4-\mu}}\Bigg( \int\limits_\Omega Q^{\frac{4}{4-\mu}}(|x|) e^{\frac{4\alpha v' |u|^2}{4-\mu}} dx\Bigg)^{\frac{1}{v'}}.
		\end{align*}
		Choosing $v$ close to 1, $\alpha$ close to $\alpha_0$ and $\|u\|=\rho$ sufficiently small such that
		$\frac{4\alpha v'}{4-\mu}\|u\|^2 \leq 4\pi \Bigg(1+\frac{2b_0}{4-\mu}\Bigg)$. Combining the above three inequalities with Lemma \ref{lem2.2}, we can find constants $C_3, C_4 >0$ such that
		\begin{align}
			\int\limits_\Omega Q^{\frac{4}{4-\mu}}(|x|)|F(u)|^{\frac{4}{4-\mu}}dx &\leq C_3 \|u\|^2+ C_4 \|u\|^{\frac{4q}{4-\mu}}. \label{2}
		\end{align}
		Using Proposition \ref{Hardy}, \eqref{2} and by characterization of $\lambda_1$ given in \eqref{1.6}, we can obtain
		\begin{align}
			J(u) &= \frac{1}{2} \|u\|^2 - \frac{\lambda}{2} \|u\|^2_2 - \1\int\limits_\Omega\Bigg( \int\limits_\Omega \frac{Q(|y|)F(u(y))}{|x-y|^\mu}dy\Bigg) Q(|x|)F(u(x))dx\nonumber\\
			& \geq \frac{1}{2} \|u\|^2 - \frac{\lambda}{2\lambda_1} \|u\|^2_2 - \frac{ C_{\mu}}{2} \Bigg(\int\limits_\Omega Q^{\frac{4}{4-\mu}}(|x|) |F(u(x))|^{\frac{4}{4-\mu}} \Bigg)^{\frac{4-\mu}{2}} \nonumber\\
			& \geq  \frac{1}{2}\Bigg(1-\frac{\lambda}{\lambda_1}\Bigg) \|u\|^2 - C_5 \|u\|^{4-\mu} -C_6 \|u\|^{2q}. \label{3}
		\end{align}
		Next, we denote $g(\rho)=C\rho^2-C_5\rho^{4-\mu}-C_6\rho^{2q}$. Observe that $g(\rho) = 0$ when $\rho = 0$ and $4-\mu>2$, $2q>2$. Hence, for sufficiently small $\rho>0$, there exists $\delta>0$ such that $g(\rho) \geq \delta >0$.
	\end{proof}
	\begin{prop}
		\label{prop4.3}
		Assume \eqref{H1}-\eqref{AR}. Then there exists $v \in \hr$ with $\|v\|>\rho$ such that $J(v)<0$.
	\end{prop}
	\begin{proof}
		Fix $u_0 \in \hr\setminus \{0\}$. For each $t>0$, we define
		\begin{align}
			\psi(t)= I\Big( \frac{tu_0}{\|u_0\|}\Big), \text{ where } I(u)= \frac{1}{2}\int\limits_\Omega \Bigg( \int\limits_\Omega \frac{Q(|y|)F(u(y))}{|x-y|^\mu}dy\Bigg) Q(|x|)F(u(x))dx. \nonumber
		\end{align}
		Then using assumption \eqref{AR}, we have
		\begin{align}
			\frac{\psi'(t)}{\psi(t)} &= \frac{\psi'(t)t}{\psi(t)t}\geq \frac{I'\Big( \frac{tu_0}{\|u_0\|}\Big)\frac{tu_0}{\|u_0\|}}{I\Big( \frac{tu_0}{\|u_0\|}\Big)t}\nonumber\\
			&\geq \frac{ \displaystyle\int\limits_\Omega \Big( \int\limits_\Omega \frac{Q(|y|)F\Big(u\Big( \frac{tu_0(y)}{\|u_0\|}\Big)\Big)}{|x-y|^\mu}dy\Big) Q(|x|)f\Big(u\Big( \frac{tu_0(y)}{\|u_0\|}\Big)\Big) u \Big( \frac{tu_0(x)}{\|u_0\|}\Big)dx}{\displaystyle\1 \displaystyle\int\limits_\Omega \Big( \int\limits_\Omega \frac{Q(|y|)F\Big(u\Big( \frac{tu_0(y)}{\|u_0\|}\Big)\Big)}{|x-y|^\mu}dy\Big) Q(|x|)F\Big(u\Big( \frac{tu_0(x)}{\|u_0\|}\Big)\Big)dx}\nonumber\geq \frac{2K}{t}, \ \ \forall \ \ t>0. \nonumber
		\end{align}
		Integrating the above integral over $[1, s\|u_0\|]$ with $s> \frac{1}{\|u_0\|}$, then
		\begin{align}
			\log(\psi(s\|u_0\|))-\log(\psi(1))&\geq 2K (\log(s\|u_0\|)).\nonumber
		\end{align}
		By simple calculation, we have
		\begin{align}
			I(su_0) \geq I\Big(\frac{u_0}{\|u_0\|}\Big) \|u_0\|^{2K}s^{2K}. \label{1}
		\end{align}
		By the definition of the functional $J$ and from \eqref{1}, we can get
		\begin{align}
			J(su_0) &\leq \frac{s^2}{2}\|u_0\|^2-\frac{\lambda s^2}{2} \|u_0\|^2_2- I\Big(\frac{u_0}{\|u_0\|}\Big)\|u_0\|^{2K}s^{2K}\nonumber\\
			&\leq C_1s^2-C_2s^2-C_3s^{2K}, \nonumber
		\end{align}
		which implies that $J(su_0) \rightarrow -\infty$ as $s\rightarrow \infty$. Therefore, for $v=su_0$ with $s$ large enough, the result holds.
	\end{proof}
	
	\subsection{The minimax level} The eventual loss of compactness of the Palais-Smale sequence occurs when dealing with critical problems. We need to bring the level $c$ below a threshold value to restore compactness. This requires us to prove that $J$ satisfies the Palais-Smale condition at the level $c$. To reestablish the compactness of $\{u_n\}$, we first calculate the upper bound of the mountain-pass level $c$. Inspired from \cite{ Adimurthi1, Cassani, Figueiredo}, let us consider the Moser's functions defined as follows:
	\begin{equation}
		\nonumber
		M_n(x) = \frac{1}{\sqrt{2\pi}}
		\left\{
		\begin{aligned}
			\sqrt{\log{n}}, && |x|< \frac{1}{n}\\
			\frac{\log{(1/|x|)}}{\sqrt{\log{n}}},  &&\frac{1}{n} \leq |x| < 1\\
			0, && |x| \geq 1
		\end{aligned}
		\right.
	\end{equation} 
	
	By simple computations, we see that $M_n$ satisfies the following estimates:
	\begin{align*}
		\|M_n\|^2_2 &= \int\limits_{B_1} |M_n|^2 dx= \frac{1}{2\pi} \int\limits_0^{2\pi} \int\limits_0^{\frac{1}{n}} \log{n}r dr + \frac{1}{2\pi}\int\limits_0^{2\pi} \int\limits_{\frac{1}{n}}^1 \frac{\log^2{1/r}}{\log{n}}rdr=  \frac{1}{4\log{n}}-\frac{1}{4n^2\log{n}}- \frac{1}{2n^2}.
	\end{align*}
	Now,
	\begin{align*}
		\|M_n\|^2 &= \int\limits_{B_1} |\nabla M_n|^2 dx= \frac{1}{2\pi} \int\limits_0^{2\pi} \int\limits_{\frac{1}{n}}^1 \Bigg( \frac{r}{\sqrt{\log{n}}} \cdot \frac{-1}{r^2}\Bigg)^2 r dr =  \int\limits_{\frac{1}{n}}^1 \frac{1}{r\log{n}}dr = 1.
	\end{align*}
	\begin{prop}
		\label{prop5}
		Assume \eqref{H1}, \eqref{d}, \eqref{H4}. Then, mountain-pass level $c< \frac{(4-\mu)\pi}{2\alpha_0}\Big( 1+ \frac{2b_0}{4-\mu}\Big)$.
	\end{prop}
	\begin{proof}
		Because for $u_0\not \equiv0$, $J(su_0)\to -\infty$ as $s \to \infty$ by Proposition \ref{prop4.3}, and since $c \leq \displaystyle \max_{[0,1]} J(tu_0)$ for $u_0\in \hr \setminus \{0\}$ satisfying $J(u_0) < 0$, it is sufficient to find a $w\in \hr$ such that $\|w\|=1$ and
		\[\max_{s\geq 0} J(sw_n)< \frac{(4-\mu)\pi}{2\alpha_0}\Big( 1+ \frac{2b_0}{4-\mu}\Big).
		\] 
		For this, consider the sequence of Moser functions $\{M_n\}$ defined as above. We claim that there exists $n\in \mathbb{N}$ such that $\displaystyle\max_{s\geq 0} J(sw_n)< \frac{(4-\mu)\pi}{2\alpha_0}\Big( 1+ \frac{2b_0}{4-\mu}\Big)$. Suppose by contradiction, the above inequality does not hold. Then, this maximum is larger than or equal to $\frac{(4-\mu)\pi}{2\alpha_0}\Big( 1+ \frac{2b_0}{4-\mu}\Big)$. Let $s_n >0$ be such that 
		\begin{equation}
			\label{4}
			J(s_nw_n) = \max_{s \geq 0} J(sw_n).
		\end{equation}
		Then $J(s_nw_n) \geq \frac{(4-\mu)\pi}{2\alpha_0}\Big( 1+ \frac{2b_0}{4-\mu}\Big), \forall n \in \mathbb{N}.$ This implies that
		\begin{equation}
			\label{5}
			s_n^2 \geq \frac{(4-\mu)\pi}{2\alpha_0}\Big( 1+ \frac{2b_0}{4-\mu}\Big).
		\end{equation}
		From \eqref{4}, $\frac{d}{ds}J(sw_n)\left|_{s=s_n}=0 \right.$. We multiply by $s_n$ and using the above proposition, we have
		\begin{align}
			s_n^2 -\lambda s_n^2 - \int\limits_\Omega \Bigg( \int\limits_\Omega \frac{Q(|y|)F(s_nw_n(y))}{|x-y|^\mu}dy \Bigg) Q(|x|)f(s_nw_n(x))s_n w_n(x)dx =0. \nonumber
		\end{align} 
		This further implies that 
		\begin{align}
			s_n^2 \geq \int\limits_\Omega \Bigg( \int\limits_\Omega Q(|y|)\frac{F(s_nw_n(y))}{|x-y|^\mu}dy \Bigg) Q(|x|)f(s_nw_n(x))s_n w_n(x)dx. \label{6}
		\end{align}
		Combining \eqref{d} and \eqref{H4}, for all $\varepsilon \in (0, \beta_0)$, there exists $R_\varepsilon>0$ such that
		\[ F(s)f(s)s \geq M_0^{-1} (\beta_0-\varepsilon)^2 s^{v+1}e^{2\alpha_0 s^2}, \ \forall s \geq R_\varepsilon.\] 
		From \eqref{6} and the above inequality implies that, for large $n$, we obtain
		\begin{align}
			s_n^2 &\geq \int\limits\limits_{B_{1/n}(0)} \Bigg( \int\limits\limits_{B_{1/n(0)}} \frac{Q(|y|)F(s_nw_n(y))}{|x-y|^\mu}dy \Bigg) Q(|x|)f(s_nw_n(x))s_n w_n(x)dx\nonumber\\
			& \geq F\Bigg( \frac{s_n \sqrt{\log{n}}}{\sqrt{2\pi}}\Bigg) f\Bigg( \frac{s_n \sqrt{\log{n}}}{\sqrt{2\pi}}\Bigg) \frac{s_n \sqrt{\log{n}}}{\sqrt{2\pi}} \int\limits\limits_{B_{1/n}(0)} \int\limits\limits_{B_{1/n}(0)} \frac{Q(|x|)Q(|y|)}{|x-y|^\mu}dxdy\nonumber\\
			& \geq M_0^{-1}(\beta_0-\varepsilon)^2 \Bigg( \frac{\log{n}}{2\pi}\Bigg)^{\frac{v+1}{2}} s_n^{v+1} e^{\alpha_0 s_n^2 \pi^{-1}\log{n}} \int\limits\limits_{B_{1/n}(0)} \int\limits\limits_{B_{1/n}(0)} \frac{|x|^{b_0} |y|^{b_0}}{|x-y|^\mu}\nonumber\\
			& = M_0^{-1}(\beta_0-\varepsilon)^2  \Bigg( \frac{\log{n}}{2\pi}\Bigg)^{\frac{v+1}{2}} \frac{s_n^{v+1} e^{\alpha_0 s_n^2 \pi^{-1}\log{n}}}{n^{2b_0}} \int\limits\limits_{B_{1/n}(0)} \int\limits\limits_{B_{1/n}(x)} \frac{dzdx}{|z|^\mu}\nonumber\\
			& \geq M_0^{-1}(\beta_0-\varepsilon)^2 \Bigg( \frac{\log{n}}{2\pi}\Bigg)^{\frac{v+1}{2}} \frac{s_n^{v+1} e^{\alpha_0 s_n^2 \pi^{-1}\log{n}}}{n^{2b_0}} \int\limits\limits_{B_{1/n}(0)} \int\limits\limits_{B_{\frac{1}{n}-|x|}(0)} \frac{dzdx}{|z|^\mu}\nonumber\\
			& \geq M_0^{-1}(\beta_0-\varepsilon)^2 \Bigg( \frac{\log{n}}{2\pi}\Bigg)^{\frac{v+1}{2}} \frac{4\pi^2}{(2-\mu)(3-\mu)(4-\mu)} \frac{s_n^{v+1} e^{\alpha_0 s_n^2 \pi^{-1}\log{n}}}{n^{4+2b_0-\mu}} \nonumber
		\end{align}
		By some simple computations, we have
		\begin{align}
			s_n^2	& \geq  M_0^{-1}(\beta_0-\varepsilon)^2 \Bigg( \frac{1}{2\pi}\Bigg)^{\frac{v+1}{2}} \frac{4\pi^2}{(2-\mu)(3-\mu)(4-\mu)} s_n^{v+1} e^{[\alpha_0 s_n^2 \pi^{-1}-(4+2b_0-\mu)]\log{n}+ \frac{v+1}{2}\log\log{n}}. \label{7}
		\end{align}
		For large $n$, we have $\log\log{n}>0$, so we can ignore this term. Then,
		\begin{align}
			(1-v)\log{s_n} & \geq  \log\left[M_0^{-1}(\beta_0-\varepsilon)^2 \Bigg( \frac{1}{2\pi}\Bigg)^{\frac{v+1}{2}} \frac{4\pi^2}{(2-\mu)(3-\mu)(4-\mu)} \right] \nonumber \\
			& \quad+ [\alpha_0 s_n^2 \pi^{-1}-(4+2b_0-\mu)]\log{n}. \label{8}
		\end{align}
		If $\{s_n\}$ is unbounded, up to a subsequence $s_n \rightarrow +\infty$ as $n\rightarrow +\infty$ and then 
		\begin{align}
			\frac{(1-v)\log s_n }{s_n^2} & \geq s_n^{-2} \log\left[M_0^{-1}(\beta_0-\varepsilon)^2 \Bigg( \frac{1}{2\pi}\Bigg)^{\frac{v+1}{2}} \frac{4\pi^2}{(2-\mu)(3-\mu)(4-\mu)} \right] \nonumber \\
			& \quad+ [\alpha_0 \pi^{-1}-s_n^{-2}(4+2b_0-\mu)]\log{n} \nonumber
		\end{align}
		gives a contradiction. Therefore, passing to a subsequence, there exists a positive constant $s_0$ such that
		\[\lim_{n\rightarrow \infty} s_n^2 = s_0^2 \geq \frac{(4-\mu)\pi}{\alpha_0} \Big( 1+ \frac{2b_0}{4-\mu}\Big).\]
		Moreover, $s_0= \frac{(4-\mu)\pi}{\alpha_0}\Big( 1+ \frac{2b_0}{4-\mu}\Big)$, otherwise a contradiction occurred by \eqref{8}. Let's take $n\rightarrow\infty$ in \eqref{7}, this implies a contradiction. Therefore, $c < \frac{(4-\mu)\pi}{2\alpha_0} \Big( 1+\frac{2b_0}{4-\mu}\Big)$ holds.
	\end{proof}
	
	\begin{lem}
		\label{lem4.6}
		Assume \eqref{H1} and \eqref{d}. Then, there exists a function $u_0\in H^1_{0,rad}(\Omega) \setminus \{0\}$ such that $J'(u_0)=0$.
	\end{lem}
	\begin{proof}
		Let $\{u_n\}$ be a Palais-Smale sequence in $\hr$, then by Lemma \ref{lem3.1}, $\{ \|u_n\|\}$ is uniformly bounded in $n\in \mathbb{N}$.  This implies that, passing to a subsequence, there exists a function $u_0\in \hr$ such that $u_n \rightharpoonup u_0$ in $\hr$, $u_n \rightarrow u_0$ in $L^p(\Omega)$, with $p \geq 1$ and $u_n(x) \rightarrow u_0(x)$ a.e. in $\Omega$.\\
		We claim that $u_0\neq0$. On the contrary, suppose $u_0 \equiv 0$. According to $J'(u_n)u_n \rightarrow 0$, we conclude that \eqref{3.14} holds true for some $C_0>0$ and then from Lemma \ref{lem3.2}, we have
		\begin{align}
			&\lim_{n\rightarrow \infty} \int\limits_\Omega \Big( \int\limits_\Omega \frac{Q(|y|)F(u_n(y))}{|x-y|^\mu}dy\Big)Q(|x|) F(u_n(x))dx \nonumber\\
			& \quad \ \ \ = \int\limits_\Omega \Big( \int\limits_\Omega \frac{Q(|y|)F(u_0(y))}{|x-y|^\mu}\Big)dy Q(|x|) F(u_0(x))dx=0 \label{4.9}
		\end{align}
		which together with \eqref{3.4} and Proposition \ref{prop5} implies that 
		\begin{align}
			\label{4.10}
			\lim_{n\rightarrow \infty} \|u_n\|^2 = 2c < \frac{(4-\mu)\pi}{\alpha_0} \Big( 1+\frac{2b_0}{4-\mu}\Big).
		\end{align}
		From Proposition \ref{Hardy} and equation \eqref{3.1}, we have
		\begin{align}
			&\Bigg| \int\limits_\Omega \Bigg( \int\limits_\Omega \frac{Q(|y|)f(u_n(y))u_n(y)}{|x-y|^\mu}dy \Bigg) Q(|x|) f(u_n(x)) u_n(x) dx\Bigg|\nonumber\\
			& \quad \leq C_\mu \Bigg( \int\limits_\Omega Q^{\frac{4}{4-\mu}}(|x|) |f(u_n(x)) u_n(x)|^{\frac{4}{4-\mu}}dx\Bigg)^{\frac{4-\mu}{2}}\nonumber\\
			& \quad \leq C_1 \Bigg( \int\limits_\Omega Q^{\frac{4}{4-\mu}}(|x|) |u_n|^2dx\Bigg)^\frac{4-\mu}{2} + C_2 \Bigg( \int\limits_\Omega Q^{\frac{4}{4-\mu}}(|x|) |u_n|^{\frac{4q}{4-\mu}} e^{\frac{4\alpha |u_n|^2}{4-\mu}} dx\Bigg)^{\frac{4-\mu}{2}}.\nonumber
		\end{align}
		Putting $p=\frac{4-\mu}{2}$ in Lemma \ref{lem3}, we get
		\begin{align}
			\label{4.11}
			&\Bigg| \int\limits_\Omega \Bigg( \int\limits_\Omega \frac{Q(|y|)f(u_n(y))u_n(y)}{|x-y|^\mu}dy \Bigg) Q(|x|) f(u_n(x)) u_n(x) dx\Bigg|\nonumber\\
			& \quad \leq C_3 \|u_n\|^{4-\mu} + C_2 \Bigg(\int\limits_\Omega Q^{\frac{4}{4-\mu}}(|x|) |u_n|^{\frac{4q}{4-\mu}} e^{\frac{4\alpha |u_n|^2}{4-\mu}}dx \Bigg)^{\frac{4-\mu}{2}}. 
		\end{align}
		Now, we choose $\alpha > \alpha_0$ sufficiently close to $\alpha_0$ and $\frac{4qv}{4-\mu} \geq 2$ such that $\frac{1}{v}+\frac{1}{v'}=1$ and $\frac{4\alpha v' \|u_n\|^2}{4-\mu} \leq 4\pi \Big(1+\frac{2b_0}{4-\mu}\Big)$. With this choice, by H\"older's inequality and Lemma \ref{lem2.2},
		\begin{align}
			&\Bigg(\int\limits_\Omega Q^{\frac{4}{4-\mu}}(|x|) |u_n|^{\frac{4q}{4-\mu}} e^{\frac{4\alpha |u_n|^2}{4-\mu}} dx \Bigg)^{\frac{4-\mu}{2}}\nonumber\\
			& \quad \ \ \ \leq \Bigg( \int\limits_\Omega Q^{\frac{4}{4-\mu}}(|x|) |u_n|^{\frac{4qv}{4-\mu}}dx \Bigg)^{\frac{4-\mu}{2v}} \Bigg( \int\limits_\Omega Q^{\frac{4}{4-\mu}}(|x|) e^{\frac{4\alpha v'|u_n|^2}{4-\mu}} dx \Bigg)^{\frac{4-\mu}{2v'}}\nonumber\\
			& \quad \ \ \  \leq C \|u_n\|^{2q}. \nonumber
		\end{align}
		Using the above estimate in \eqref{4.11}, we obtain a constant $C>0$ such that
		\begin{align}
			&\Bigg| \int\limits_\Omega \Bigg( \int\limits_\Omega \frac{Q(|y|)f(u_n(y))u_n(y)}{|x-y|^\mu}dy \Bigg) Q(|x|) f(u_n(x)) u_n(x) dx\Bigg| \leq C_3 \|u_n\|^{4-\mu} + C_4\|u_n\|^{2q} \leq C. \nonumber
		\end{align}
		Applying Cauchy-Schwarz inequality with \eqref{4.9}, 
		\begin{align*}
			&\Bigg| \int\limits_\Omega \Bigg( \int\limits_\Omega \frac{Q(|y|)F(u_n(y))}{|x-y|^\mu}dy \Bigg) Q(|x|) f(u_n(x)) u_n(x) dx\Bigg| \\
			&\leq \Bigg[ \int\limits_\Omega \Bigg( \int\limits_\Omega \frac{Q(|y|)f(u_n(y))u_n(y)}{|x-y|^\mu}dy \Bigg) Q(|x|) f(u_n(x)) u_n(x) dx\Bigg]^{\1} \times\\
			& \quad \quad \Bigg[ \int\limits_\Omega \Bigg( \int\limits_\Omega \frac{Q(|y|)F(u_n(y))}{|x-y|^\mu}dy \Bigg) Q(|x|) F(u_n(x)) dx\Bigg]^{\1}\\
			& \quad \rightarrow 0 \text{ as } n \rightarrow \infty.
		\end{align*}
		This together with $J'(u_n)u_n \to 0$ implies that $\|u_n\| \to 0$, which further implies from \eqref{4.10} that $c=0$. But the mountain-pass value $c\neq 0$, hence we get a contradiction. Therefore $u_0 \neq 0$. Furthermore, from Lemma \ref{lemma3.4} and $J'(u_n)u_n \rightarrow 0$, we get $J'(u_0)=0$.
	\end{proof}
	\begin{lem}
		\label{lem4.7}
		Assume \eqref{H1} and \eqref{AR}. Then the functional $J$ satisfies $(PS)_c$ condition for all $c< \frac{(4-\mu)\pi}{2\alpha_0}\Big( 1+\frac{2b_0}{4-\mu}\Big).$ 
	\end{lem}
	\begin{proof}
		Suppose $\{u_n\}$ is a Palais-Smale sequence in $\hr$. As we have seen in Lemma \ref{lem4.6}, there exists a $u_0 \in \hr\setminus\{0\}$ such that $u_n \rightharpoonup u_0$ in $\hr$, $u_n \to u_0$ in $L^p(\Omega)$ with $1 \leq p < \infty$ and $u_n \to u_0 $ a.e. in $\Omega$.  For each $n\in \mathbb N$, we define $v_n$ and $v$ as
		\begin{align*}
			v_n = \frac{u_n}{\|u_n\|} , \text{ and }  v=\frac{u_0}{\displaystyle \lim_{n\to \infty} \|u_n\|}.
		\end{align*}
		By Fatou's lemma, $0 < \|u_0\| \leq \displaystyle \liminf_{n \to \infty} \|u_n\|$, this implies that $0< \|v\| \leq 1$. If $\|v\|=1$, then $\|u_n\| \to \|u_0\|$ which together with $u_n \rightharpoonup u_0$ in $\hr$ implies that $u_n \to u_0$ in $\hr$. In this case, the proof is done. So, we assume $0< \|v\|<1$. Since $J(u_n) \to c$, from Lemma \ref{lem3.2}, we have
		\begin{align}
			1-\|v\|^2 &= 1- \frac{\|u_0\|^2}{\displaystyle\lim_{n\to \infty}\|u_n\|^2}\nonumber\\
			& = \frac{2c+ \lambda \|u_0\|^2_2 + \displaystyle\int\limits_\Omega \Bigg( \int\limits_\Omega \frac{Q(|y|)F(u_0(y))}{|x-y|^\mu}dy \Bigg) Q(|x|)F(u_0(x))dx - \|u_0\|^2}{{\displaystyle\lim_{n\to \infty}}\|u_n\|^2}.\nonumber
		\end{align}
		This implies that
		\begin{align}
			\label{4.12}
			\lim_{n\to \infty}\|u_n\|^2 = \frac{2c-2J(u_0)}{1-\|v\|^2}.
		\end{align}
		From Lemma \ref{lem4.6}, $J'(u_0) = 0$, then we have
		\begin{align}
			J(u_0)= J(u_0) - \1 J'(u_0)u_0 = \1 \int\limits_\Omega \Bigg( \int\limits_\Omega \frac{Q(|y|)F(u_0(y))}{|x-y|^\mu}dy\Bigg)Q(|x|) [f(u_0(x)) u_0(x)- F(u_0(x))] dx, \nonumber
		\end{align}
		which together with \eqref{AR}, we can say that $J(u_0) \geq 0$.
		Using this inequality with \eqref{4.12}, we obtain
		\begin{align}
			\lim_{n\to \infty} \|u_n\|^2 \leq \frac{2c}{1-\|v\|^2} < \frac{(4-\mu)\pi}{\alpha_0(1-\|v\|^2)}\Big( 1+\frac{2b_0}{4-\mu}\Big). \nonumber
		\end{align}
		Choosing $\alpha>\alpha_0$ sufficiently close to $\alpha_0$ and $p>1$ close to 1 such that $\frac{1}{p}+\frac{1}{p'} =1$ and \begin{align}
			\frac{4\alpha p \|u_n\|^2}{4-\mu} < \frac{4\pi}{1-\|v\|^2}\Bigg( 1+\frac{2b_0}{4-\mu}\Bigg).\nonumber
		\end{align}
		Then from above inequality and Lemma \ref{lem2.3}, we have
		\begin{align}
			\label{4.13}
			\sup_{n \in \mathbb{N}} \int\limits_\Omega Q^{\frac{4}{4-\mu}}(|x|)e^{\frac{4\alpha p}{4-\mu}|u_n|^2}dx &= \sup_{n \in \mathbb{N}} \int\limits_\Omega Q^{\frac{4}{4-\mu}}(|x|)e^{\frac{4\alpha p}{4-\mu}\|u_n\|^2|v_n|^2}dx < +\infty. 
		\end{align}
		Next, we claim that 
		\begin{align}
			\label{4.14}\int\limits_\Omega \Bigg( \int\limits_\Omega \frac{Q(|y|)F(u_n(y))}{|x-y|^\mu}dy\Bigg) Q(|x|) f(u_n(x)) (u_n(x)-u_0(x))dx \to 0.
		\end{align}
		To prove this claim, we use Proposition \ref{Hardy}, 
		\begin{align}
			\label{4.15}
			&\int\limits_\Omega \Bigg( \int\limits_\Omega \frac{Q(|y|)F(u_n(y))}{|x-y|^\mu}dy\Bigg) Q(|x|) f(u_n(x)) (u_n(x)-u_0(x))dx \\
			&\leq C_\mu \Bigg(\int\limits_\Omega Q^{\frac{4}{4-\mu}}(|x|) |F(u_n)|^{\frac{4}{4-\mu}}dx\Bigg)^{\frac{4-\mu}{4}} \Bigg(\int\limits_\Omega Q^{\frac{4}{4-\mu}}(|x|) |f(u_n)(u_n-u_0)|^{\frac{4}{4-\mu}} dx\Bigg)^{\frac{4-\mu}{4}}.\nonumber
		\end{align}
		We estimate first integral using \eqref{3.2}, Lemma \ref{lem3} and \eqref{4.13}, we can find a constant $\overline{C}$ satisfying 
		\begin{align}
			&\int\limits_\Omega Q^{\frac{4}{4-\mu}}(|x|) |F(u_n)|^{\frac{4}{4-\mu}}dx \nonumber\\
			&\leq C \int\limits_\Omega Q^{\frac{4}{4-\mu}}(|x|) |u_n|^2 dx + C \int\limits_\Omega Q^{\frac{4}{4-\mu}}(|x|) |u_n|^{\frac{4q}{4-\mu}} e^{\frac{4\alpha}{4-\mu}|u_n|^2}dx\nonumber\\
			& \leq C\|u_n\|^2 +C \Bigg( \int\limits_\Omega Q^{\frac{4}{4-\mu}}(|x|) |u_n|^{\frac{4qp'}{4-\mu}} dx\Bigg)^{\frac{1}{p'}} \Bigg( \int\limits_\Omega Q^{\frac{4}{4-\mu}}(|x|) e^{\frac{4\alpha p}{4-\mu}|u_n|^2}dx\Bigg)^{\frac{1}{p}}\nonumber\\
			& \leq C\|u_n\|^2 +C \|u_n\|^{\frac{4q}{4-\mu}} \Bigg( \int\limits_\Omega Q^{\frac{4}{4-\mu}}(|x|) e^{\frac{4\alpha p}{4-\mu}|u_n|^2}dx\Bigg)^{\frac{1}{p}}\nonumber\leq \overline{C} <+\infty. \nonumber
		\end{align}
		On the other hand, from \eqref{3.1} with $q>1$, Lemma \ref{lem3}, \eqref{4.13} and taking $q>2$
		\begin{align}
			&\int\limits_\Omega Q^{\frac{4}{4-\mu}}(|x|) |f(u_n)(u_n-u_0)|^{\frac{4}{4-\mu}} dx\nonumber\\
			&\leq C \int\limits_\Omega Q^{\frac{4}{4-\mu}}(|x|) |u_n|^{\frac{2(2-\mu)}{4-\mu}}|u_n-u_0|^{\frac{4}{4-\mu}}dx + C \int\limits_\Omega Q^{\frac{4}{4-\mu}}(|x|) |u_n-u_0|^{\frac{4}{4-\mu}}|u_n|^{\frac{4(q-1)}{4-\mu}} e^{\frac{4\alpha}{4-\mu}|u_n|^2}dx\nonumber\\
			&\leq C \Bigg(\int\limits_\Omega Q^{\frac{4}{4-\mu}}(|x|) |u_n|^{2}dx\Bigg)^{\frac{2-\mu}{4-\mu}} \Bigg(\int\limits_\Omega Q^{\frac{4}{4-\mu}}(|x|) |u_n-u_0|^{2}dx\Bigg)^{\frac{2}{4-\mu}}\nonumber\\
			&\quad +C \Bigg( \int\limits_\Omega Q^{\frac{4}{4-\mu}}(|x|) |u_n-u_0|^{\frac{4p'}{4-\mu}} |u_n|^{\frac{4(q-1)p'}{4-\mu}}dx\Bigg)^{\frac{1}{p'}} \Bigg( \int\limits_\Omega Q^{\frac{4}{4-\mu}}(|x|) e^{\frac{4\alpha p}{4-\mu}|u_n|^2}dx\Bigg)^{\frac{1}{p}}\nonumber\\
			& \leq C\|u_n\|^{\frac{2(2-\mu)}{4-\mu}} \|u_n-u_0\|^{\frac{4}{4-\mu}} \nonumber\\
			& \quad+ C \Bigg( \int\limits_\Omega Q^{\frac{4}{4-\mu}}(|x|) |u_n-u_0|^{\frac{8p'}{4-\mu}}dx\Bigg)^{\frac{1}{2p'}}\Bigg( |u_n|^{\frac{8(q-1)p'}{4-\mu}}dx\Bigg)^{\frac{1}{2p'}}\Bigg( \int\limits_\Omega Q^{\frac{4}{4-\mu}}(|x|) e^{\frac{4\alpha p}{4-\mu}|u_n|^2}dx\Bigg)^{\frac{1}{p}}\nonumber\\
			& \to 0 \text{ as } n \to \infty.\nonumber
		\end{align}
		Then by the above two estimates and \eqref{4.15}, we conclude the proof of claim \eqref{4.14}. This implies that $J'(u_n)(u_0-u_n) \to 0$. Then applying convexity of the functional $I(u) = \frac{\|u\|^2}{2}$ to obtain
		\begin{align}
			\1\|u_0\|^2 = I(u_0) &\geq I(u_n) + I'(u_n)(u_0-u_n)\nonumber\\
			& = I(u_n) + \int\limits_\Omega \nabla u_n \nabla (u_0-u_n)dx\nonumber
		\end{align}
		and we have 
		\begin{align}
			\1\|u_0\|^2 
			&\geq \1 \|u_n\|^2 + J'(u_n)(u_0-u_n) +\lambda \int\limits_\Omega u_n(u_0-u_n)dx\nonumber\\
			&\quad - \int\limits_\Omega \Bigg( \int\limits_\Omega \frac{Q(|y|)F(u_n)}{|x-y|^\mu}dy\Bigg) Q(|x|)f(u_n)(u_n-u_0)dx.\nonumber
		\end{align}
		Hence, we have $\displaystyle\lim_{n\to \infty} \|u_n\|^2 \leq \|u_0\|^2$. This together with Fatou's lemma implies that $\displaystyle\lim_{n\to \infty} \|u_n\|^2 = \|u_0\|^2$. Since $u_n \rightharpoonup u_0$ in $\hr$, therefore $u_n \to u_0$ in $\hr$. This completes the proof.
	\end{proof}
	\begin{proof}[\bf Proof of Theorem \ref{thm1.1}]
		Using Propositions \ref{prop4.2} and Proposition \ref{prop4.3}, there exists a Palais-Smale sequence $\{u_n\}$ for $J$ in $\hr$, which is bounded by Lemma \ref{lem3.1}. Since $\hr$ is a reflexive, there exists $u_0\in \hr$ such that $\{u_n\}\rightharpoonup u_0$ in $\hr$. From Lemma \ref{lemma3.4}, $u_0$ is a critical point of the functional in $\hr$. Next, we aim to show that $u_0$ is nontrivial. In fact, if $u_0$ is a mountain-pass critical point, then as mountain-pass level $c>0$, implies that $u_0 \not\equiv 0$. Therefore, the only thing remaining to show is that $u_0$ is a mountain-pass critical point. From Lemma \ref{lem4.7} we deduce that $J(u_0)=c$ in $\hr$, where $c$ is mountain-pass level. Since all the assumptions of Theorem \ref{thm4.1} are satisfied, it follows that there exists a nontrivial critical point $u_0$ in $\hr$. To complete the proof, we cite Theorem \ref{thm1.3} (see Appendix for the verification), and establish that $u_0$ is indeed a weak solution of equation \eqref{1.1}.
	\end{proof}
	\section{Linking case where $\lambda_k < \lambda < \lambda_{k+1}$}
	\label{sec5}
	When $\lambda>\lambda_1$, Theorem \ref{thm4.1}-(i) no longer holds, thus our previous existence approach fails. In such cases, we use the following critical point theorem known as the Linking theorem, due to A. Ambrosetti and
	P. Rabinowitz \cite{Ambrosetti}, which provides a milder version of Theorem \ref{thm4.1}-(i).
	\begin{thm}
		\label{thm5.1}
		Let $J:\mathcal{H} \rightarrow \mathbb{R}$ be a $C^1$ functional on a Banach space $(\mathcal{H}, \|\cdot\|)$ such that $\mathcal{H}=\mathcal{H}_1\oplus \mathcal{H}_2$ with $dim \mathcal{H}_1 < \infty$. If $J$ satisfies the following:
		\begin{enumerate}
			\item [(i)] There exist constants $\rho, \delta >0$ such that $J(u) \geq \delta$ for all $u\in \mathcal{H}_2$ satisfying $\|u\|=\rho$.
			\item [(ii)] There exists $z \notin \mathcal{H}_1$ with $\|z\|=1$ and $R>\rho$ such that $J(u) \leq 0$ for all $u \in \partial \mathcal Q$, where 
			\[ \mathcal Q =\{ v+sz : v \in \mathcal{H}_1, \|v\|\leq R \text{ and } 0 \leq s \leq R\}.\]
			\item [(iii)] There exists some $\beta>0$ such that $J$ satisfies the $(PS)_c$ for $c\in (0,\beta)$.
			Then $c$ is defined as 
			\[ c= \inf_{\gamma\in \Gamma} \max_{u\in \mathcal Q} J(\gamma(u)), \]
			where $\Gamma = \{ \gamma \in C(\overline{\mathcal Q}, \mathcal{H}): \gamma(u)=u, \text{ if } u \in \partial \mathcal Q$\}, is a critical value of $J$.
		\end{enumerate}
	\end{thm}
	In the following propositions, we show the above geometry.
	\begin{prop}
		\label{prop5.2}
		Let $\lambda_k < \lambda< \lambda_{k+1}$ and $f$ satisfies \eqref{H1}. Then there exists $\delta, \rho >0$ such that $J(u) \geq \delta$, for $\|u\|=\rho$ and $u\in H_{k,r}^\perp(\Omega)$.
	\end{prop}
	\begin{proof}
		Let $u\in H_{k,r}^\perp(\Omega)$, proceeding as similar to \eqref{3} and using the characterization of $\lambda_{k+1}$ given in \eqref{1.7}, we have
		\begin{align}
			J(u) &\geq \frac{1}{2} \|u\|^2 -\frac{\lambda}{2} \|u\|^2_2 -C_5\|u\|^{4-\mu} -C_6 \|u\|^{2q}\nonumber\\
			& \geq \frac{1}{2} \Bigg(1 - \frac{\lambda}{\lambda_{k+1}}\Bigg)\|u\|^2- C_5\|u\|^{4-\mu} -C_6 \|u\|^{2q}\nonumber\\
			& \geq C\|u\|^2- C_5\|u\|^{4-\mu} -C_6 \|u\|^{2q},\nonumber
		\end{align}
		this implies that for sufficiently small $\|u\|=\rho>0$, there exists $\delta>0$ such that $J(u) \geq \delta$.
	\end{proof}
	\begin{prop}
		\label{prop5.3}
		Let $\lambda_k < \lambda < \lambda_{k+1}$ and \eqref{H1}, \eqref{AR} holds. Define $\mathcal Q = \{ v+sz : v\in H_{k,r}(\Omega), \|v\| \leq R \text{ and } 0 \leq s \leq R, \text{ for some } R> \rho\}$, where $\rho$ is given in Proposition \ref{prop5.2} and $z \in H_{k,r}^\perp(\Omega)$ with $\|z\|=1$. Then $J(u) \leq 0$ for all $u\in \partial \mathcal Q$.
	\end{prop}
	\begin{proof}
		For some $R>0$, we split $\partial \mathcal Q$ into following three parts:\\
		$$ \mathcal Q_1 =\{ u\in H_{k,r}(\Omega): \|u\| \leq R\}$$
		$$ \mathcal Q_2= \{ u+sz: u\in H_{k,r}(\Omega), \|u\|=R \text{ and } 0\leq s\leq R\}$$
		$$\mathcal Q_3 = \{ u+Rz: u\in H_{k,r}(\Omega), \|u\|\leq R\}.$$
		Case 1. If $u\in \mathcal Q_1$, this implies $u\in H_{k,r}(\Omega)$ and by characterization of $\lambda_k$ given in \eqref{1.8} together with the fact that $F(s)\geq 0$, we have
		\begin{equation*}
			J(u) \leq \frac{1}{2}\Bigg( 1-\frac{\lambda}{\lambda_k}\Bigg)\|u\|^2 -\frac{1}{2} \int\limits\limits_\Omega \Bigg( \int\limits\limits_\Omega \frac{Q(|y|)F(u(y))}{|x-y|^\mu}dy\Bigg) Q(|x|)F(u(x))dx\leq \frac{1}{2}\Bigg( 1-\frac{\lambda}{\lambda_k}\Bigg)R^2<0,
		\end{equation*}
		for any choice of $R>0$.\\
		Before verifying the claim on $\mathcal Q_2, \mathcal Q_3$, let us fix some $u_0 \in H_{k,r}(\Omega)$ and define a map $\phi: \R \rightarrow \R$ as $\phi(t) = J(tu_0)$. From \eqref{1}, we have
		\begin{equation*}
			\phi(t) \leq \frac{t^2}{2} \|u_0\|^2 -\frac{\lambda t^2}{2} \|u_0\|^2_2 -C t^{2K} \|u_0\|^{2K}, \ \text{where } K>1,
		\end{equation*}
		which implies that $\phi(t) \rightarrow-\infty$ as $t\rightarrow\infty$.\\
		Case 2. If $u\in \mathcal Q_2$, there exists $v\in H_{k,r}(\Omega)$ and $0\leq s \leq R$ such that $u=v+sz$. Moreover,
		\begin{align*}
			\|u\|^2 = \|v+sz\|^2 = \|v\|^2 + s^2\|z\|^2 \geq \|v\|^2 =R^2.
		\end{align*}
		Therefore, if we choose $R>0$ sufficiently large, we have $J(u)<0$. \\
		Case 3. Now, if $u\in \mathcal Q_3$, then there exists $v\in H_{k,r}(\Omega)$ such that $u=v+Rz$. Moreover,
		\begin{equation*}
			\|u\|^2 = \|v+Rz\|^2 = \|v\|^2 +R^2\|z\|^2 \geq R^2.
		\end{equation*}
		For $R>0$ large enough we have $J(u)<0$. 
	\end{proof}
	\subsection{The minimax level}
	We have to choose a $z\in H_{k,r}^\perp(\Omega)$ such that $\|z\|=1$ and $J(u) < \frac{(4-\mu)\pi}{2\alpha_0}\Big( 1+\frac{2b_0}{4-\mu}\Big), \ \forall \ u \in Q.$ Let $P_k: \hr \rightarrow H_{k,r}^\perp(\Omega)$ be the orthogonal projection. Let us define 
	\begin{equation} \label{5.1}
		W_n(x) = P_k(M_n(x))
	\end{equation} 
	Since $H_{k,r}$ is a finite dimensional subspace, then there exists $A_0>0$ and $B_0>0$ such that
	\begin{equation}
		\label{5.2}
		\left\{ 
		\begin{aligned}
			&\|u\|^2 \leq A_0 \|u\|^2_2 \ \text{ and }\\
			&\|u\|_\infty \leq \frac{B_0}{B} \|u\|_2 
		\end{aligned} 
		\  \ \ \forall\ \ \ u\in H_k,
		\right.
	\end{equation}
	where $B >0$ such that $\|M_n\|_2 \leq \frac{B}{\sqrt{\log{n}}}$ \ $\forall \ n\in \mathbb{N}$.
	\begin{lem}
		\label{lem5.4}
		Let $W_n$ be defined in \eqref{5.1}. Then the following estimates hold:
		\begin{enumerate}
			\item [(i)] $1 - \frac{A_0}{\log{n}} \leq \|W_n\|^2 \leq 1$.
			\item [(ii)] $W_n(x) \geq \left\{ 
			\begin{aligned}
				&\frac{-B_0}{\sqrt{\log{n}}},  &&\ \Omega\setminus B_{\frac{1}{n}}(0)\\
				& \frac{1}{\sqrt{2\pi}}\sqrt{\log{n}}- \frac{B_0}{\sqrt{\log{n}}}, && \ B_{\frac{1}{n}}(0)
			\end{aligned}\right. $
		\end{enumerate}
	\end{lem}
	\begin{proof} By simple computations, we have
		\begin{align}
			\|W_n\|^2 &= \|M_n\|^2 - \|(I-P_k)M_n\|^2 \ \text{ and } (I-P_k)M_n \in H_{k,r}(\Omega) \nonumber\\
			& \leq \|M_n\|^2 = 1. \nonumber 
		\end{align}
		On the other hand, from \eqref{5.2}, we have $\|(I-P_k)M_n\|^2 \leq A_0 \|(I-P_k)M_n\|^2_2$. So,
		\begin{align}
			\label{5.3}
			\|W_n\|^2 &\geq \|M_n\|^2 - A_0 \|(I-P_k)M_n\|_2^2 \nonumber\\
			&= 1 - A_0\|(I-P_k)M_n\|^2_2.
		\end{align}
		By simple calculations, we have
		\begin{align}
			\|M_n\|^2_2 & = \|P_k(M_n)+ (I-P_k)M_n\|^2_2\nonumber\\
			&= \|P_k(M_n)\|^2_2 + \|(I-P_k)M_n\|_2^2\nonumber\\
			& \geq \|(I-P_k)M_n\|^2_2 \nonumber.
		\end{align}
		Using this in \eqref{5.3}, we have
		\begin{align*}
			\|W_n\|^2 &\geq 1 - A_0\|M_n\|^2_2\nonumber\\
			& \geq 1 - A_0 \Bigg( \frac{1}{4\log{n}}- \frac{1}{4n^2\log{n}}- \frac{1}{2n^2}\Bigg) \nonumber\\
			& \geq 1 - \frac{A_0}{\log{n}}.
		\end{align*}
		This completes the proof of $(i)$.
		As $M_n \geq 0$ in $\Omega$ and $M_n = \frac{1}{\sqrt{2\pi}}\sqrt{\log{n}}$ in $|x|<\frac{1}{n}$, we have
		\begin{align*}
			W_n(x) &= M_n(x)- (I-P_k)(M_n(x))\\
			&\geq \left\{ \begin{aligned}
				&-\|(I-P_k)M_n\|_\infty , && \mbox{if} \ \Omega\setminus B_{\frac{1}{n}}(0) \\
				& \frac{1}{\sqrt{2\pi}}\sqrt{\log{n}} - \|(I-P_k)M_n\|_\infty, && \mbox{if}\  B_{\frac{1}{n}}(0)
			\end{aligned}\right.\\
			&\geq \left\{ \begin{aligned}
				&-\frac{B_0}{B}\|(I-P_k)M_n\|_2 , && \mbox{if} \  \Omega\setminus B_{\frac{1}{n}}(0)\\
				& \frac{1}{\sqrt{2\pi}}\sqrt{\log{n}} - \frac{B_0}{B}\|(I-P_k)M_n\|_2, && \mbox{if}\  B_{\frac{1}{n}}(0)
			\end{aligned}\right.
		\end{align*}	
		where we used \eqref{5.2}. Since $\|(I-P_k)M_n\|_2 \leq \|M_n\|_2$ and by the definition of $B$, i.e. $\|M_n\|_2 \leq \frac{B}{\sqrt{\log{n}}}$, we can get
		\begin{align*}
			W_n(x)	&\geq \left\{ \begin{aligned}
				&-\frac{B_0}{\sqrt{\log{n}}} , && \mbox{if} \ \Omega\setminus B_{\frac{1}{n}}(0)\\
				& \frac{1}{\sqrt{2\pi}}\sqrt{\log{n}} - \frac{B_0}{\sqrt{\log{n}}}, && \mbox{if}\  B_{\frac{1}{n}}(0)
			\end{aligned}\right.
		\end{align*}
	\end{proof}
	
	We define $z_n(x)=\frac{W_n(x)}{\|W_n(x)\|},\ \mathcal Q_n = \{ v+sz_n: v\in H_{k,r}(\Omega), \|v\|\leq R \text{ and } 0 \leq s \leq R, \text{ for some } R> \rho \}$ and the minimax level of $J$ as follows:
	\begin{equation}
		\label{5.4}
		0<c(n) =\inf_{\gamma \in \Gamma_n} \max_{w\in \mathcal Q_n} J(\gamma(w)),
	\end{equation}
	where $\Gamma_n = \{ \eta \in C(\overline{\mathcal Q}_n, H): \eta(w)=w \text{ if } w \in \partial \mathcal Q_n\}$.
	
	The following proposition is a crucial estimate for critical problems. To ensure compactness properties for the functional 
	$J$, the minimax value obtained from the linking geometry must not exceed a specific constant. Specifically, we can locate a critical point of $J$ at the level $c(n)$ if we establish that 
	$c(n)< \frac{(4-\mu)\pi}{2\alpha_0}\Big(1+\frac{2b_0}{4-\mu}\Big).$
	In Proposition \ref{prop5}, we have proved the minimax level of the mountain-pass type using the sequence $\{s_n\}$. However, in this case, we need to consider a sequence of the type $\{v_n+s_n z_n\}$, where $v_n \in H_{k,r}$. The presence of $v_n$ makes the arguments more complicated than in the previous case. Precisely, we have the following bound on $c(n)$.
	
	\begin{prop}
		Let $c(n)$ be given as in \eqref{5.4} and assumptions \eqref{H1}-\eqref{H4} hold. Then for some $n$, $c(n)< \frac{(4-\mu)\pi}{2\alpha_0}\Big(1+\frac{2b_0}{4-\mu}\Big).$
	\end{prop}
	\begin{proof}Based on the definition of $c(n)$ and $id\in \Gamma_n$, it suffices to prove that $\max\{ J(v+sz_n): v \in H_{k,r}(\Omega), \|v\| \leq R, 0 \leq s \leq R\} < \frac{(4-\mu)\pi}{2\alpha_0}\Big(1+\frac{2b_0}{4-\mu}\Big)$. 
		On the contrary, let us assume this condition is not satisfied. Then for all $n \in \mathbb{N}$
		\[ \max\{ J(v+sz_n): v\in H_{k,r}, \|v\| \leq R, 0 \leq s \leq R\} \geq \frac{(4-\mu)\pi}{2\alpha_0}\Big(1+\frac{2b_0}{4-\mu}\Big).\]
		But note that for $s \geq R$, we have $J(v+sz_n)\leq0$. Hence, due to the compactness of $H_{k,r} \cap \overline{B_R}$, for each $n$, there exist $s_n>0$ and $v_n \in H_{k,r}$ such that
		\[ J(v_n+s_n z_n) = \max\{ J(v+sz_n): v\in H_{k,r}, \|v\| \leq R, 0 \leq s \leq R\}.\]
		This further implies that
		\begin{equation}
			\label{5.5}
			J(v_n+s_nz_n) \geq \frac{(4-\mu)\pi}{2\alpha_0}\Big(1+\frac{2b_0}{4-\mu}\Big).
		\end{equation}
		Let $u_n = v_n+s_nz_n$, then $J'(u_n)=0$ and we get
		\begin{equation}
			\label{5.6}
			\|u_n\|^2 -\lambda \|u_n\|^2_2 - \int\limits_\Omega \Bigg( \int\limits_\Omega \frac{Q(|y|)F(u_n(y))}{|x-y|^\mu}dy\Bigg) Q(|x|)f(u_n(x))u_n(x)dx =0.
		\end{equation}
		We complete the proof in the following two steps:\\
		Step 1. In this step, we prove that $\{v_n\}$ and $\{s_n\}$ are bounded sequences. 
		
		We have the following two distinct possibilities:
		\begin{enumerate}
			\item [(i)] $\frac{s_n}{\|v_n\|} \geq C_0,$ for some $C_0>0$ uniformly in $n$.
			\item [(ii)] $\frac{s_n}{\|v_n\|} \rightarrow 0$ in $\R$, up to a subsequence as $n \rightarrow \infty$. 
		\end{enumerate}
		Suppose $(i)$ holds. Then, the boundedness of $\{s_n\}$ implies the boundedness of $\{v_n\}$ as $\|v_n\| \leq \frac{s_n}{C_0}.$ Thus, it is sufficient to prove that the sequence $\{s_n\}$ is bounded. Condition $(i)$ implies that there exists a constant $C>0$ such that
		\begin{equation}
			\|u_n\| = \|v_n+s_nz_n\| \leq \|v_n\|+s_n\|z_n\| \leq \frac{s_n}{C_0} + s_n \leq \sqrt{C}s_n. \nonumber
		\end{equation}
		Using this in equation \eqref{5.6}, we get
		\begin{align}
			Cs_n^2 &\geq \int\limits\limits_\Omega \Big( \int\limits\limits_\Omega \frac{Q(|y|)F(u_n(y))}{|x-y|^\mu}dy\Big)Q(|x|)f(u_n(x)) u_n(x)dx \nonumber\\
			& \geq \int\limits\limits_{B_{\frac{1}{n}}} \Big( \int\limits\limits_{B_\frac{1}{n}} \frac{Q(|y|)F(u_n(y))}{|x-y|^\mu}dy\Big) Q(|x|)f(u_n(x))u_n(x)dx. \nonumber
		\end{align}
		The assumption \eqref{d} implies that, there exists $M_0>0$ and $s_0>0$ such that $f(s)s \geq M_0^{-1}s^{v+1} F(s)$ for all $s\geq s_0$. From \eqref{H4}, there exists $\varepsilon \in (0, \beta_0)$ such that $F(s) \geq (\beta_0-\varepsilon)e^{\alpha_0 s^2}$ for all $s \geq s_\varepsilon$. Choosing $\overline{s} = \max\{s_0, s_\varepsilon\}$, we have
		\begin{align}
			Cs_n^2 
			& \geq \frac{ (\beta_0-\varepsilon)^2}{M_0}\int\limits\limits_{B_{\frac{1}{n}} \cap \{ |u_n| \geq \overline{s}\}} \Big( \int\limits\limits_{B_\frac{1}{n} \cap \{ |u_n| \geq \overline{s}\}} \frac{|y|^{b_0}e^{\alpha_0 |u_n(y)|^2}}{|x-y|^\mu} dy\Big) |x|^{b_0}|u_n(x)|^{v+1} e^{\alpha_0 |u_n(x)|^2}dx. \label{5.7}
		\end{align}
		To estimate the above integral, we use Lemma \ref{lem5.4} $(i)$, $\|W_n(x)\|^2 \leq 1$ and from $(ii)$, in $B_{\frac{1}{n}}$ for $n$ large enough, we have
		\begin{align*}
			u_n(x) &= v_n(x) + s_n z_n(x) = v_n(x) + s_n \frac{W_n(x)}{\|W_n(x)\|}\\
			& \geq s_n \Big( \frac{1}{\sqrt{2\pi}}\sqrt{\log{n}} - \frac{B_0}{\log{n}}\Big) \Bigg( \frac{v_n(x)}{s_n\Big(\frac{1}{\sqrt{2\pi}}\sqrt{\log{n}}- \frac{B_0}{\log{n}}\Big)} + 1\Bigg)\\
			& \geq s_n \Big( \frac{1}{\sqrt{2\pi}}\sqrt{\log{n}} - \frac{B_0}{\log{n}}\Big) \Bigg( 1- \frac{ \|v_n\|_\infty}{s_n\Big(\frac{1}{\sqrt{2\pi}}\sqrt{\log{n}}- \frac{B_0}{\log{n}}\Big)} \Bigg).
		\end{align*}
		Since $H_{k,r}$ is a finite dimensional subspace, then there is a $C>0$ satisfying $\|v\|_\infty \leq C \|v\|$ for all $v\in H_{k,r}$. This implies that there exist $\delta', \delta'' \in (0,1)$, such that 
		\begin{align*}
			u_n(x)& \geq s_n \Big( \frac{1}{\sqrt{2\pi}}\sqrt{\log{n}} - \frac{B_0}{\log{n}}\Big)\Bigg( 1- \frac{ C\|v_n\|}{s_n\Big(\frac{1}{\sqrt{2\pi}}\sqrt{\log{n}}- \frac{B_0}{\log{n}}\Big)} \Bigg)\\
			& \geq s_n \Big( \frac{1}{\sqrt{2\pi}}\sqrt{\log{n}} - \frac{B_0}{\log{n}}\Big) \Bigg( 1- \frac{ C}{C_0\Big(\frac{1}{\sqrt{2\pi}}\sqrt{\log{n}}- \frac{B_0}{\log{n}}\Big)} \Bigg)\\
			& \geq \frac{s_n \sqrt{\log{n}}}{\sqrt{2\pi}} \Bigg( 1- \frac{ C}{C_0\Big(\frac{1}{\sqrt{2\pi}}\sqrt{\log{n}}- \frac{B_0}{\log{n}}\Big)} \Bigg)- \frac{B_0}{\log{n}} \Bigg( 1- \frac{ C}{C_0\Big(\frac{1}{\sqrt{2\pi}}\sqrt{\log{n}}- \frac{B_0}{\log{n}}\Big)} \Bigg)\\
			& \geq \frac{s_n \sqrt{\log{n}}}{\sqrt{2\pi}} (1-\delta') - \delta''.
		\end{align*}
		Choosing $ \delta \in (\delta', 1)$ such that $\frac{s_n \sqrt{\log{n}}}{\sqrt{2\pi}} (\delta-\delta') > \delta''$ for $n$ is large enough, then we have
		\begin{align*}
			u_n(x) &\geq \frac{s_n \sqrt{\log{n}}}{\sqrt{2\pi}} (1-\delta) + \frac{s_n \sqrt{\log{n}}}{\sqrt{2\pi}} (\delta-\delta') - \delta''\\
			& \geq \frac{s_n \sqrt{\log{n}}}{\sqrt{2\pi}} (1-\delta).
		\end{align*}
		
		Since $\frac{s_n \sqrt{\log{n}}}{\sqrt{2\pi}}( 1- \delta)$ is arbitrary large, we can take $\frac{s_n \sqrt{\log{n}}}{\sqrt{2\pi}}( 1- \delta) > \overline{s}$ for large $n$. Therefore, $B_{\frac{1}{n}} \subset \Big\{|u_n| \geq \frac{s_n \sqrt{\log{n}}}{\sqrt{2\pi}} ( 1- \delta) \Big\}$.
		From equation \eqref{5.7},
		\begin{align}
			Cs_n^2 &\geq \frac{ (\beta_0-\varepsilon)^2}{M_0 n^{2b_0}}\int\limits\limits_{B_{\frac{1}{n}}} \Big( \int\limits\limits_{B_\frac{1}{n}} \frac{e^{\frac{\alpha_0(1-\delta)^2 s_n^2 \log{n}}{2\pi}}}{|x-y|^\mu} dy\Big) \Big( \frac{s_n \sqrt{\log{n}}}{\sqrt{2\pi}} (1-\delta)\Big)^{v+1} e^{\frac{\alpha_0(1-\delta)^2 s_n^2 \log{n}}{2\pi}}dx \nonumber\\
			&\geq \frac{(\beta_0-\varepsilon)^2 e^{\frac{\alpha_0(1-\delta)^2 s_n^2 \log{n}}{\pi}}}{M_0 n^{2b_0}}\Big( \frac{s_n \sqrt{\log{n}}}{\sqrt{2\pi}} (1-\delta)\Big)^{v+1} \int\limits\limits_{B_{\frac{1}{n}}} \Big( \int\limits\limits_{B_\frac{1}{n}} \frac{1}{|x-y|^\mu} dy\Big) dx\nonumber  \\
			&\geq \frac{4\pi^2}{(2-\mu)(3-\mu)(4-\mu)} M_0^{-1} (\beta_0-\varepsilon)^2 e^{\frac{\alpha_0(1-\delta)^2 s_n^2 \log{n}}{\pi}}	\Big( \frac{s_n \sqrt{\log{n}}}{\sqrt{2\pi}} (1-\delta)\Big)^{v+1} n^{\mu-2b_0-4}. \nonumber
		\end{align}
		We get a constant $C_0= C_0(\mu, M_0, \overline{s}, \beta_0, \varepsilon)>0$ satisfying
		\begin{align}
			& s_n^2\geq C_0 (1-\delta)^{v+1} s_n^{v+1} e^{[\alpha_0 (1-\delta)^2 s_n^2 \pi^{-1}-(4+2b_0-\mu)]\log{n} + \frac{v+1}{2} \log{\log{n}}}. \label{5.8}
		\end{align}
		Since $ \frac{v+1}{2} \log{\log{n}} >0$, then
		\begin{equation}
			\label{5.9}
			s_n^2 \geq C_0 (1-\delta)^{v+1} s_n^{v+1} e^{[\alpha_0 (1-\delta)^2 s_n^2 \pi^{-1}-(4+2b_0-\mu)]\log{n}}. 
		\end{equation}
		By simple computations, we have 
		\begin{align}
			\label{5.10}
			\frac{(1-v)\log{s_n}}{s_n^2} \geq & \frac{\log{C_0(1-\delta)^{v+1}}}{s_n^2} + \left[\alpha_0 (1-\delta)^2 \pi^{-1}-\frac{(4+2b_0-\mu)}{s_n^2} \right]\log{n}.
		\end{align}
		If $s_n \rightarrow \infty$ in \eqref{5.10}, then we obtain a contradiction. Hence $\{s_n\}$ is a bounded sequence, so is $\{v_n\}$.
		
		Next, we assume that $(ii)$ occurs. Then, $s_n\leq \|v_n\|$, this gives that $\|u_n\|= \|v_n+ s_n z_n\| \leq \|v_n\|+ s_n \|z_n\|\leq 2 \|v_n\|$. This implis that if the sequence $\{v_n\}$ is bounded in $\hr$, then the sequence $\{s_n\}$ is bounded in $\R$. Therefore, our goal is to prove that $\{v_n\}$ is bounded in $\hr$. Let us assume that this is not true. Then, up to a subsequence $\|v_n\| \rightarrow \infty$. Let  $d= \sup\{ |x-y|: x,y \in \Omega\}$ denotes diameter of $\Omega$. From \eqref{5.6}, we have
		\begin{align}
			1 &\geq \frac{1}{\|u_n\|^2}\int\limits\limits_\Omega \Big( \int\limits\limits_\Omega \frac{Q(|y|)F(u_n(y))}{|x-y|^\mu}dy\Big) Q(|x|)f(u_n(x))u_n(x)\ dx \nonumber\\
			& \geq \frac{1}{4d^\mu \|v_n\|^2} \int\limits\limits_{ B_{\frac{1}{n}}} \Big( \int\limits\limits_{ B_{\frac{1}{n}}} |y|^{b_0}F(u_n(y))dy\Big) |x|^{b_0}f(u_n(x))u_n(x)dx. \nonumber 
		\end{align}
		Using \eqref{d} and \eqref{H4}, we can find $\varepsilon \in (0, \beta_0)$ and $\overline{s} = \max\{s_0, s_\varepsilon\}>0$ such that
		\begin{align}
			1 & \geq \frac{\overline{s}^{v+1}(\beta_0 - \varepsilon)^2}{4M_0d^\mu n^{2b_0} \|v_n\|^2} \int\limits\limits_{ B_{\frac{1}{n}} \cap \{|u_n| \geq \overline{s}\}} \Big( \int\limits\limits_{ B_{\frac{1}{n}} \cap \{|u_n| \geq \overline{s}\}} e^{\alpha_0 u_n^2(y)}dy\Big) e^{\alpha_0 u_n^2(x)}\ dx. \label{5.11}
		\end{align}
		Note that,
		\[ \frac{u_n}{\|v_n\|} \chi_{ B_{\frac{1}{n}} \cap \{|u_n| \geq \overline{s}\}} = \frac{v_n}{\|v_n\|} + \frac{s_n}{\|v_n\|}z_n - \frac{u_n}{\|v_n\|} \chi_{ B_{\frac{1}{n}} \cap \{|u_n| < \overline{s}\}}.\]
		Since $\frac{s_n}{\|v_n\|} \rightarrow 0$ in $\R$, $z_n \rightarrow 0$ a.e. in $\Omega$ and $\frac{u_n}{\|v_n\|} \rightarrow 0$ in $ B_{\frac{1}{n}} \cap \{|u_n| < \overline{s}\}$, then there exists $v_0 \in H_{k,r}(\Omega)$ such that 
		\begin{equation*}
			\frac{u_n(x)}{\|v_n\|} \chi_{ B_{\frac{1}{n}} \cap \{ |u_n| \geq \overline{s}\}} \rightarrow \overline{v} \text{ a.e. in } \Omega,
		\end{equation*}
		with $\frac{v_n}{\|v_n\|} \rightarrow \overline{v}$ and $\|\overline{v}\|=1$. From \eqref{5.11} and $\|v_n\| \to \infty$ as $n\to \infty$, 
		\begin{equation*}
			1\geq \frac{\overline{s}^{v+1}(\beta_0 - \varepsilon)^2}{4M_0d^\mu n^{2b_0}\|v_n\|^2} \int\limits\limits_\Omega \int\limits\limits_\Omega e^{\alpha_0 \|v_n\|^2 \Bigg( \frac{u_n(x)}{\|v_n\|} \chi_{\Omega \cap \{ |u_n| \geq \overline{s}\}}\Bigg)^2} e^{\alpha_0 \|v_n\|^2 \Bigg( \frac{u_n(x)}{\|v_n\|} \chi_{\Omega \cap \{ |u_n| \geq \overline{s}\}}\Bigg)^2}dx dy\rightarrow \infty, 
		\end{equation*}
		we get a contradiction. Hence, both $\{s_n\}$ and $\{v_n\}$ are bounded. \\
		Step 2. As we have proved that $\{s_n\}$ and $\{v_n\}$ are bounded, we can assume that up to a subsequence, there exists $v_0 \in \hr$ and $s_0\in \R$ such that $v_n \rightarrow v_0$ and $s_n \rightarrow s_0$. In this step, we prove that $v_0=0$ and $s_0^2 = \frac{(4-\mu)\pi}{\alpha_0}\Big( 1+\frac{2b_0}{4-\mu}\Big)$.
		
		Observe that since $M_n \rightharpoonup 0$, we have $ z_n \rightharpoonup 0$ and this implies $z_n \to 0$ a.e. in $\Omega$. Therefore, 
		\[u_n=v_n+s_n z_n \to v_0 \text{ a.e. in } \Omega.\]
		From \eqref{5.6} and Lemma \ref{lem3.1}, we have
		\begin{align*} \int\limits\limits_\Omega \Big( \int\limits\limits_\Omega \frac{Q(|y|)F(u_n(y))}{|x-y|^\mu}dy \Big) Q(|x|)f(u_n(x))u_n(x)dx = \|u_n\|^2 - \lambda\|u_n\|^2_2\leq  \|u_n\|^2 \leq C.
		\end{align*}
		By Remark \ref{rem3.1}, we have 
		\begin{equation} 
			\label{5.12}
			\int\limits\limits_\Omega \Big( \int\limits\limits_\Omega \frac{Q(|y|)F(u_n(y))}{|x-y|^\mu}dy \Big) Q(|x|)F(u_n(x)) dx \rightarrow \int\limits\limits_\Omega \Big( \int\limits\limits_\Omega \frac{Q(|y|)F(v_0(y))}{|x-y|^\mu}dy \Big) Q(|x|)F(v_0(x))dx.
		\end{equation}
		Since $v_n \rightarrow v_0$ in $\hr$, $\|z_n\|_2 \rightarrow 0$ and $s_n\to s_0$. Therefore, from \eqref{5.5}, \eqref{5.12} and the characterization of $\lambda_k$, we obtain
		\begin{align*}
			\frac{(4-\mu)\pi}{2\alpha_0} \Big( 1+\frac{2b_0}{4-\mu}\Big)&\leq \frac{1}{2} \|v_n\|^2 + \frac{s_n^2}{2} - \frac{\lambda}{2} \|v_n\|^2_2+ \frac{\lambda}{2} s_n^2 \|z_n\|_2^2\\
			&\quad- \frac{1}{2} \int\limits\limits_\Omega \Bigg( \int\limits\limits_\Omega \frac{Q(|y|)F(v_0(y))}{|x-y|^\mu}dy \Bigg) Q(|x|)F(v_0(x))dx\\
			& = \frac{1}{2} \|v_0\|^2 + \frac{s_0^2}{2} - \frac{\lambda}{2} \|v_0\|^2_2 \\
			&\quad- \frac{1}{2} \int\limits\limits_\Omega \Bigg( \int\limits\limits_\Omega \frac{Q(|y|)F(v_0(y))}{|x-y|^\mu}dy \Bigg) Q(|x|)F(v_0(x))dx.
		\end{align*}
		By the definition of the functional $J$ given in \eqref{J}, we have
		\begin{align}
			\label{6.1}
			J(v_0) + \frac{s_0^2}{2} \geq 	\frac{(4-\mu)\pi}{2\alpha_0} \Big( 1+\frac{2b_0}{4-\mu}\Big).
		\end{align}
		Since $J(v_0) \leq 0$ for all $v_0 \in H_{k,r}$, from Proposition \ref{5.3}. We have  $s_0^2 \geq \frac{(4-\mu)\pi}{\alpha_0} \Big( 1+\frac{2b_0}{4-\mu}\Big)$ and we know that $\|v_n\|\leq C$, for some $C>0$, this implies that $(ii)$ is not possible. Hence only $(i)$ is true, then from \eqref{5.9}
		\[ s_n^2 \geq C_0(1-\delta)^{v+1} s_n^{v+1} e^{[\alpha_0 (1-\delta)^2 s_n^2 \pi^{-1}-(4+2b_0-\mu)]\log{n}}.\]
		This implies that $s_0^2 \leq \frac{(4-\mu)\pi}{\alpha_0 (1-\delta)^2} \Big(1+\frac{2b_0}{4-\mu}\Big)$, otherwise we get a contradiction as $n\to \infty$. Taking $\delta$ very small close to 0, then we have $s_0^2 \leq \frac{(4-\mu)\pi}{\alpha_0}\Big(1+\frac{2b_0}{4-\mu}\Big)$. Hence $s_0^2 = \frac{(4-\mu)\pi}{\alpha_0}\Big(1+\frac{2b_0}{4-\mu}\Big)$.\\
		Next, we show that $v_0 \equiv 0$. From \eqref{6.1}, $J(v_0) \geq 0$ but $J(v_0) \leq 0$ for all $v_0 \in H_{k,r}$. Thus, $J(v_0)=0$. 
		
		By the characterization of $\lambda_k$, we have
		\begin{align*}
			0= J(v_0) & \leq \frac{1}{2} \Big(1 - \frac{\lambda}{\lambda_k}\Big)\|v_0\|^2 - \int\limits\limits_\Omega \Bigg( \int\limits\limits_\Omega \frac{Q(|y|)F(v_0(y))}{|x-y|^\mu}dy\Bigg) Q(|x|)F(v_0(x))dx.
		\end{align*}
		Using assumption \eqref{H1} and \eqref{AR}, we have $F(s)\geq 0$ for all $s\geq 0$. This together with $\lambda > \lambda_k$, we get
		\begin{align*}
			0 \leq \frac{1}{2} \Big(1 - \frac{\lambda}{\lambda_k}\Big)\|v_0\|^2 <0,
		\end{align*}
		Hence, $\|v_0\|=0$. This completes our claim.
		
		To complete the proof, note that from Step 2, up to a subsequence, we have $v_n \rightarrow 0$ in $H_{k,r}$ and $s_n \rightarrow s_0$. Then from equation \eqref{5.8}, letting $\delta \rightarrow 0^+$,
		\begin{align*}
			s_n^2 \geq C_0 s_n^{v+1}e^{[ \alpha_0 s_n^2 \pi^{-1} - (4+2b_0-\mu)]\log{n} + \frac{v+1}{2}\log{\log{n}}}.
		\end{align*}
		Now taking $n\rightarrow \infty$, then we get a contradiction. This completes the proof of the proposition.
	\end{proof}
	\begin{proof}[\bf Proof of Theorem \ref{thm1.1}] The Propositions \ref{prop5.2} and \ref{prop5.3} are used to derive the proofs of Theorem \ref{thm5.1}-$(i)$ and $(ii)$ satisfied by the functional $J$ in $\hr$. Consequently, there exists a Palais-Smale sequence for $J$ at level $c$ in $\hr$. From Lemma \ref{lem3.1}, there exists $u_0 \in \hr$ such that, up to a subsequence, $u_n \rightharpoonup u_0$ in $\hr$. Since $c>0$, by Lemma \ref{lem4.6}, $u_0$ is nontrivial. This fact, together with Lemma \ref{lemma3.4}, implies $J'(u_0)=0$. According to Lemma \ref{lem4.7}, $J$ satisfies Theorem \ref{thm5.1}-(iii), i.e. $J(u_0)=c$ with $c < \frac{(4-\mu)\pi}{2\alpha_0}\Big( 1+\frac{2b_0}{4-\mu}\Big)$. Hence, Theorem \ref{thm5.1} implies that $u_0$ is a critical point of the functional $J$ in $\hr$. By using Theorem \ref{thm1.3} (we refer to Appendix for the proof), we conclude that $u_0$ is a critical point of $J$ in $\h$. So, we finish the proof of Theorem \ref{thm1.1}.
	\end{proof}
	\section*{Appendix}
	In this section, we show that the functional $J$ is $\mathcal{O}(2)$-invariant, where $\mathcal{O}(2)$ is defined in Section \ref{sec1}. For this, we need to assume that $\Omega$ is invariant with respect to $\mathcal{O}(2)$, then using change of variable and $|det (h^{-1})|=1$, we have
	\begin{align*}
		J(hu)  & = \int\limits_\Omega \int\limits_\Omega \frac{ Q(|y|) Q(|x|) F(hu(y)) F(hu(x))}{|x-y|^\mu}dy dx\\
		&=\int\limits_\Omega \int\limits_\Omega \frac{ Q(|y|) Q(|x|) F(u(h^{-1}y)) F(u(h^{-1}x))}{|x-y|^\mu}dy dx\\
		&= \frac{1}{|det (h^{-1})|^2} \int\limits_{H(\Omega)} \int\limits_{H(\Omega)} \frac{ Q(|h(y)|) Q(|h(x)|) F(u(y)) F(u(x))}{|h(x)-h(y)|^\mu}dy dx\\
		&=\int\limits_{\Omega} \int\limits_{\Omega} \frac{ Q(|h(y)|) Q(|h(x)|) F(u(y)) F(u(x))}{|h(x)-h(y)|^\mu}dy dx.
	\end{align*}
	Since $h$ is linear and it is isometry, then
	\[ |h(x)-h(y)|= |x-y| \text{ and } |h(x)| = |x|\]
	From the above two inequality, we conclude that 
	\begin{align*}
		J(hu) =\int\limits_{\Omega} \int\limits_{\Omega} \frac{ Q(|y)|) Q(|x)|) F(u(y)) F(u(x))}{|x-y|^\mu}dy dx = J(u).
	\end{align*}
	Hence $J$ is $\mathcal{O}(2)$-invariant, so by the principle of symmetric criticality, $u$ is a critical point of $J$ in $\h$, i.e. $u$ is a weak solution of \eqref{1.1}.

	\section*{Acknowlegement} 
	The first author acknowledges the financial aid from CSIR, Government of India, File No. 09/1237(15789)/2022-EMR-I. The research of the second author is supported by the Science and Engineering Research Board, Government of India, Grant No. MTR/2022/000495.
	
\end{document}